\documentclass[11pt]{article}
\usepackage{amssymb}
\usepackage{bm}
\usepackage{cite}
\usepackage{mathtools}
\usepackage{graphicx}
\usepackage{caption}
\usepackage{subcaption}
\usepackage{amsmath}
\usepackage{amsthm}
\usepackage[round]{natbib}
\usepackage{hyperref}
\usepackage{enumitem}

\DeclareMathOperator{\Cov}{\bf{Cov}}

\DeclareMathOperator*{\argmax}{arg\,max}

\newtheorem{theorem}{Theorem}[section]
\newtheorem{definition}[theorem]{Definition}
\newtheorem{lemma}[theorem]{Lemma}
\newtheorem{remark}[theorem]{Remark}

\newcommand{\sJM}{\mathop{\sum\nolimits\sp{\ne}}}

\begin{document}
\thispagestyle{empty}
\begin{center}
{\Large\bf Infill asymptotics for logistic regression estimators for
spatio-temporal point processes}\\[.4in]

\noindent
{\large M.N.M.~van~Lieshout$^{1,2}$ and C. Lu$^{2}$}\\[.1in]
\noindent
{\em $^{1}$Centrum Wiskunde \& Informatica \\
P.O.\ Box 94079, NL-1090 GB, Amsterdam, The Netherlands\\[.1in]
 $^{2}$Department of Applied Mathematics, University of Twente \\
 P.O.\ Box 217, NL-7500 AE, Enschede, The Netherlands
}\\[.1in]
\end{center}
\begin{verse}
{\footnotesize
\noindent
{\bf Abstract}\\
\noindent
This paper discusses infill asymptotics for
logistic regression estimators for
spatio-temporal point processes whose intensity
functions are of log-linear form. We establish
strong consistency and asymptotic normality for
the parameters of a Poisson point process model
and demonstrate how these results can be 
extended to general point process models. 
Additionally, under proper conditions, we also 
extend our central limit theorem to 
other unbiased estimating equations that are 
based on the Campbell--Mecke theorem.\\[0.1in]

\noindent
{\em Keywords \& Phrases:} Campbell--Mecke theorem, infill asymptotics, logistic regression
estimator, spatio-temporal point process, unbiased estimating equation.\\[0.1in]
}
\end{verse}

\section{Introduction}
\label{sec:introduction}

Spatial and spatio-temporal point processes have
been widely used to model, for example, earthquakes
\citep{Bray2013earthquakes}, fires
(\citealp{Moller2010wildFires};\ \citealp{Lu2021fireIncidents}) and tropical 
rain forests
\citep{Waagepetersen2008estimatingFunction}. 

Usually, the modelling procedure of a spatial or
spatio-temporal point pattern starts with
the intensity function which characterizes 
the probability of a point occurring in an
infinitesimal ball centred at a given location
and time. In many applications, the intensity
function is defined as a log-linear parametric function of certain covariates (e.g.,\ \citealp{Jesper2007parametric};\ \citealp{Waagepetersen2008estimatingFunction};\ \citealp{Guan2010weightedEstimatingFunctions};\ \citealp{Coeurjolly2014parametric}). 
To estimate the parameters, classical options include maximum likelihood estimation (e.g.,\ \citealp{Ogata1981MLE};\ \citealp{Moller2004MLE}), maximum pseudo-likelihood estimation (e.g.,\ \citealp{Besag77ISI};\ \citealp{Baddeley2000PLE}), logistic regression estimation (e.g.,\ \citealp{Baddeley2010LRPP};\ \citealp{Baddeley2014LRPP}) and minimum 
contrast estimation (e.g., \citealp{Guyon1995consistency}).

Among them, maximum likelihood estimation is the most 
computationally intensive method,
since for many point process models the 
likelihood function involves an intractable
normalizing constant that must be approximated
by Markov chain Monte Carlo simulations
(e.g., \citealp{Moller2004MLE}). Instead, 
maximum pseudo-likelihood estimation and logistic
regression estimation are based on the
well-known Campbell--Mecke and Nguyen--Zessin theorems (see, e.g.,\
\citealp{Daley2009CampbellMecke}) and are
easy to implement using standard software for
generalized linear models (e.g.,\ \citealp{Baddeley2015spatstat}). 

From a theoretical perspective, it is important
to analyze the asymptotic properties of the parameter
estimators mentioned above. For spatial and spatio-temporal point
processes, two asymptotic regimes can be
formulated: increasing-domain and infill
asymptotics \citep{Ripley1988asymptoticRegimes}.
In the former regime, the observation window
grows; whereas in the latter, the window remains
fixed but contains more and more points. 

In the literature, \citet{Waagepetersen2008estimatingFunction}
studied infill asymptotics for maximum 
pseudo-likelihood estimators for Poisson and
cluster point process models with log-linear parametric
intensity functions using various weighting
schemes to approximate the integral involved.
Again, assuming log-linear parametric intensity functions,
\citet{Thurman2015increasingDomainCI} and
\citet{Choiruddin2018increasingDomainCI}
developed increasing-domain asymptotics for
regularized versions of pseudo-likelihood 
and logistic regression estimators. In addition,
\citet{Baddeley2014LRPP} proved increasing-domain
asymptotics for logistic regression estimators
for stationary Gibbs point processes. In this paper, we study consistency and asymptotic normality for logistic regression
estimators for spatio-temporal point processes 
in the infill asymptotic regime and extend the central limit theorem to other unbiased estimating equations that are
based on the Campbell--Mecke theorem. 

The remainder of the paper is organized as
follows. Section~\ref{sec:background} introduces
some background and notation regarding
spatio-temporal point processes: parametric
intensity function, logistic regression
estimation, computations of the first two moments
of the first-order $U$-statistics and
infill asymptotic regime. Section~\ref{sec:asymptotics}
provides proofs for strong consistency and
asymptotic normality for logistic regression
estimators in the case of Poisson point process models. Section~\ref{sec:extensions}
discusses the extension of the asymptotic results to general point process models and general unbiased estimating 
equations. Finally, the paper finishes with some
conclusions.

\section{Background and notation}
\label{sec:background}

Let $X$ be a spatio-temporal point process on a bounded non-empty open set 
$W\times T \subset \mathbb{R}^{2}\times \mathbb{R}$ equipped with the Borel 
$\sigma$-algebra (see, e.g., \citealp{Daley2009CampbellMecke}). Suppose that 
the first-order moment measure of $X$, which is defined as 
\begin{equation*}
    \Lambda\left( B \right) = 
    \mathbb{E}\left\{ \sum_{x\in X} 1\left(x\in B \right) \right\}
\end{equation*}
for any Borel subset $B \subset W\times T$, exists as a $\sigma$-finite measure 
that is absolutely continuous with respect to Lebesgue measure with Radon--Nikodym 
derivative $\lambda$. Here, $1(\cdot)$ denotes the indicator function. Then 
$\lambda: W\times T \to [0,\infty)$ is called the intensity function of $X$. 

Higher order moment measures
are defined analogously. For 
instance, the second-order factorial moment
measure of $X$ is
\begin{equation*}
\Lambda^{(2)}\left( B_1 \times B_2 \right) = 
\mathbb{E}\left\{ \sJM_{x, y \in X} 
  1\left(x\in B_1, y\in B_2 \right) \right\}
\end{equation*}
for any Borel subsets $B_1, B_2 \subset W\times T$. Suppose this measure also exists as a $\sigma$-finite measure that is absolutely continuous with respect to Lebesgue measure with Radon–Nikodym derivative $\lambda^{(2)}$. Then $\lambda^{(2)}:(W\times T)^{2}\to [0,\infty)$ is called the second-order product density function of $X$. Upon scaling, one obtains the pair correlation function
$g(x,y) = \lambda^{(2)}\left(x,y\right) / (\lambda\left(x\right)\lambda\left(y\right))$ of $X$,
provided that 
$\lambda\left(x\right) \lambda\left(y\right) 
> 0$.
 
\subsection{Parametric intensity function}
\label{subsec:parametricintensity}

In this paper, we assume that $X$ admits an intensity function $\lambda$ that is of log-linear form offset by a measurable function $b$ and parameterized by a vector $\boldsymbol\theta$ in some parameter space $\boldsymbol\Theta \subset \mathbb{R}^m$. Specifically,
\begin{equation}
    \lambda\left( u;\boldsymbol\theta \right) = 
    b\left(u\right) \exp\left( \boldsymbol\theta^{\top} \bm{z}\left(u\right) \right),
    \label{e:intensity}
\end{equation}
where $u=(w,t) \in W \times T$ denotes a location and time combination, $b: W\times T \to [0,\infty)$ is a measurable function that serves as the baseline or reference intensity, $\bm{z}=[z^{(1)},\dots,z^{(m)}]^\top : W \times T \to \mathbb{R}^{m}$ is an $m$-dimensional measurable vector of spatio-temporal covariates and $\boldsymbol\theta = [\theta^{(1)},\dots,\theta^{(m)}]$ is the parameter vector. The gradient vector of 
$\lambda\left( u;\boldsymbol\theta \right)$ with respect to
$\boldsymbol\theta$ then takes the form 
\begin{equation*}
    \nabla\lambda\left(u;\boldsymbol\theta\right)
    = 
    \lambda\left(u;\boldsymbol\theta\right) \bm{z}\left(u\right).
\end{equation*}

Conditions must be imposed on $b$ and $\bm{z}$ to ensure that $\lambda$ is absolutely integrable. In the sequel, it will sometimes also be necessary to assume that $b$, and thus $\lambda$, is strictly positive. When this is the case, we will state it explicitly.

\subsection{Logistic regression estimation}
\label{subsec:LRestimation}

Estimating equations for the parameters of a spatio-temporal point process model in general and the logistic regression estimator in particular are based on the Campbell--Mecke theorem (see, e.g.,\ \citealp{Daley2009CampbellMecke}).

Consider a spatio-temporal point process $X$ on $W\times T$ with intensity function $\lambda$. For any real-valued measurable function $f$ defined on $W\times T$ such that $f\lambda$ is absolutely integrable, the Campbell--Mecke theorem reads
\begin{equation}
    \mathbb{E}\left\{ \sum_{x\in X} f\left(x\right)\right\} = 
    \int_{W \times T} f\left(u\right) \lambda\left(u\right) du,
    \label{e:campbell-mecke}
\end{equation}
where $x$ runs through the points of $X$. When $\lambda$ is parameterized by a vector $\boldsymbol\theta$ as $\lambda\left(u;\boldsymbol\theta\right)$, (\ref{e:campbell-mecke}) provides a basis for estimating $\boldsymbol\theta$. 

The logistic regression estimator is based on the components of the vector function
\begin{equation}
    \boldsymbol{f}\left(u;\boldsymbol\theta\right) = \nabla\log\left[
       \frac{  \lambda\left(u;\boldsymbol\theta\right) }
     { \lambda\left(u;\boldsymbol\theta\right) + \rho\left(u\right) }
       \right] = 
          \frac{ \rho\left(u\right) / \lambda\left(u;\boldsymbol\theta\right) }
          { \lambda\left(u;\boldsymbol\theta\right) + \rho\left(u\right) } 
          \nabla\lambda\left(u;\boldsymbol\theta\right).
    \label{e:LR_base}
\end{equation}
Here, one assumes that $\lambda(u;\boldsymbol\theta)$ is a positive-valued differentiable function such that its gradient vector $\nabla\lambda(u;\boldsymbol\theta)$ is absolutely integrable, and that $\rho(u)$ is a positive-valued measurable function also defined on $W\times T$. The idea is then to estimate both sides of (\ref{e:campbell-mecke}) and solve the equations for $\boldsymbol\theta$. In order to approximate the right-hand side, one may use a `dummy' point process $D$ on $W\times T$ which is independent of $X$ and has absolutely integrable intensity function $\rho$. Applying the Campbell--Mecke theorem to $D$, one finds that
\begin{equation*}
  \sum_{x\in D} \frac{ 1 }{ \lambda(x;\boldsymbol\theta) + \rho(x) }
  \nabla\lambda(x;\boldsymbol\theta)
\end{equation*}
is an unbiased estimator for the right-hand side of (\ref{e:campbell-mecke}) with $\boldsymbol{f}$ as in
(\ref{e:LR_base}). Hence,
 \begin{equation}
    {\boldsymbol s}(X,D;\boldsymbol\theta) = \sum_{x\in X}
       \frac{\rho(x)/\lambda(x;\boldsymbol\theta)}{\lambda(x;\boldsymbol\theta)+\rho(x)}
           \nabla\lambda(x;\boldsymbol\theta)
  - \sum_{x\in D} \frac{1}{\lambda(x;\boldsymbol\theta)+\rho(x)} 
           \nabla\lambda(x;\boldsymbol\theta) 
           = \boldsymbol{0}
    \label{e:LR_Score}
\end{equation}
is an unbiased estimating equation. It is interesting to observe that the middle part of (\ref{e:LR_Score}) is exactly the gradient of a logistic log-likelihood function
\begin{equation}
    l(X,D;\boldsymbol\theta) = \sum_{x\in X} \log\left[
          \frac{\lambda(x;\boldsymbol\theta)}{\lambda(x;\boldsymbol\theta)+\rho(x)}
    \right] + \sum_{x\in D}\log\left[
          \frac{\rho(x)}{\lambda(x;\boldsymbol\theta)+\rho(x)}
    \right].
    \label{e:LR_likelihood}
\end{equation}
Thus, (\ref{e:LR_Score}) can also be interpreted as a score function and solved using standard software such as the R-package \textit{stats} \citep{Venables2002stats}. The existence
and uniqueness of a maximizer of the log-likelihood function (\ref{e:LR_likelihood}) are ensured under proper conditions \citep{Silvapulle1981LRestimates}.

For simplicity of writing, in the remainder of this paper, we write $\boldsymbol{s}(\boldsymbol\theta)$ and $l(\boldsymbol\theta)$ for $\boldsymbol{s}(X,D;\boldsymbol\theta)$ and $l(X,D;\boldsymbol\theta)$ to suppress the dependence on the point patterns $X$ and $D$. Moreover, we use $\boldsymbol{\theta_0}$ and $\widehat{\boldsymbol\theta}$ to respectively denote the true value and the estimator of $\boldsymbol\theta$.

\subsection{First two moments of first-order \texorpdfstring{$U$}{}-statistics}
\label{subsec:momentmeasures}

Following \citet{Reitzner2013CLT}, we call random vectors of the form
\begin{equation*}
    \bm{H} = \sum_{x\in X} \left[ h^{(1)}(x), \dots, h^{(m)}(x) \right]^\top
\end{equation*}
first-order $U$-statistics of $X$ if $h^{(l)}\lambda$ is absolutely integrable for every $1 \leq l \leq m$. 

Suppose that $X$ has intensity function $\lambda^{(1)} =
\lambda$ and second-order product density function $\lambda^{(2)}$. By the Campbell--Mecke theorem and in analogy to (\ref{e:campbell-mecke}),
\begin{equation*}
    \mathbb{E}\left\{ \sJM_{x_{1},\dots,x_{r}\in X} 
       f\left(x_{1},\dots,x_{r}\right)\right\} = 
    \int_{(W\times T)^{r}} f\left(u_{1},\dots,u_{r}\right) 
       \lambda^{(r)}\left(u_{1},\dots,u_{r}\right) du_{1}\dots du_{r}
\end{equation*}
for $r=1, 2$ and any real-valued measurable function $f$ defined on $W\times T$ such that $f\lambda$ and $f\lambda^{(2)}$ are absolutely integrable. Then, it can be used to derive the first two moments of $\bm{H}$. Indeed,
\begin{equation*}
    \mathbb{E}\left\{ \bm{H} \right\} =  
    \left[\int_{W\times T} h^{(l)}(u)\lambda\left(u\right) du \right]_{l=1}^{m}
\end{equation*}
and 
\begin{equation*}
    \mathbb{E}\left\{ \bm{H} \bm{H}^{\top} \right\} =  
    \left[
    \int_{W\times T} h^{(k)}\left(u\right)h^{(l)}\left(u\right) \lambda\left(u\right) du +
    \int_{W\times T}\int_{W\times T} h^{(k)}(u)h^{(l)}(v) \lambda^{(2)}\left(u,v\right) dudv
    \right]_{k,l=1}^{m}.
\end{equation*}

If $\mathbb{E}\left\{ \bm{H} \bm{H}^{\top} \right\}$ is finite and $\lambda>0$, the entries in the covariance matrix of $\bm{H}$ are finite and can be expressed in terms of the pair correlation function $g(u,v)$ as
\begin{equation*}
\begin{split}
    \Cov\left\{ \bm{H} \right\} 
    &=  
    \left[
    \int_{W\times T} h^{(k)}\left(u\right)h^{(l)}\left(u\right)  \lambda\left(u\right) du
    \right]_{k,l=1}^{m}\\
    &+
    \left[\int_{W\times T}\int_{W\times T} h^{(k)}(u)h^{(l)}(v) (g(u,v)-1) \lambda\left(u\right) \lambda\left(v\right) dudv
    \right]_{k,l=1}^{m}.
\end{split}
\end{equation*}

\subsection{Infill asymptotic regime}
\label{subsec:infillasymptotics}

The infill asymptotic regime considered in this paper is as defined in 
\citet{Waagepetersen2008estimatingFunction} and \citet{Lieshout2020infillAsymptotics}. 

Let $\{Y_i\}$ and $\{E_i\}$, with $i \in \mathbb{N}^{+}$, be two sequences of independent and identically distributed spatio-temporal point processes with intensity functions
$\lambda$ and $\rho$, respectively. Set 
\begin{equation*}
    X_n = \bigcup_{i=1}^{n} Y_i, \quad D_n = \bigcup_{i=1}^{n} E_i.
\end{equation*}
Write $\lambda_n$ for the intensity function of $X_n$ and $\rho_n$ for that of the `dummy' point process $D_n$. Thus, $\lambda_{n} = n \lambda$
and $\rho_{n} = n \rho$, and we assume that the intensity functions of $X_{n}$ and $D_{n}$ increase at the same rate. For any $n\in \mathbb{N}^{+}$, the score function (\ref{e:LR_Score}) based on $X_n$ and $D_n$ then becomes
\begin{equation}
\begin{split}
    \boldsymbol{s}_{n}(\boldsymbol\theta) & =  \sum_{x\in X_{n}} 
    \frac{\rho_{n}(x)/\lambda_{n}(x;\boldsymbol\theta)}
            {\lambda_{n}(x;\boldsymbol\theta) + \rho_{n}(x)}
      \nabla\lambda_{n}(x;\boldsymbol\theta)
    -  \sum_{x\in D_{n}}
    \frac{1}{\lambda_{n}(x;\boldsymbol\theta) + \rho_{n}(x)} 
    \nabla\lambda_{n}(x;\boldsymbol\theta)\\
    & = \sum_{x\in X_{n}}
    \frac{\rho(x)/\lambda(x;\boldsymbol\theta)}
           {\lambda(x;\boldsymbol\theta) + \rho(x)}
     \nabla\lambda(x;\boldsymbol\theta)
     - \sum_{x\in D_{n}}
     \frac{1}{\lambda(x;\boldsymbol\theta) + \rho(x)}
     \nabla\lambda(x;\boldsymbol\theta) = 
     \boldsymbol{0}.
    \label{e:LR_score_n}
\end{split}
\end{equation}
Note that the terms in the two sums above do not depend on $n$ after simplification while the subscripts $x$ run through the points of $X_{n}$ and $D_{n}$. Taking the limit as $n\to \infty$, one obtains an asymptotic regime that \citet{Ripley1988asymptoticRegimes} calls `infill asymptotics'.

Intuitively, under this regime, the estimate for the parameter vector $\boldsymbol\theta$ will become more precise when the points observed in the fixed domain $W\times T$ become more dense. 

Our aim in the remainder of this paper is to analyze the
asymptotic behaviour of the estimator $\widehat{\boldsymbol{\theta}_{n}}$ defined by (\ref{e:LR_score_n}) as $n\rightarrow\infty$.

\section{Infill asymptotics for Poisson point processes}
\label{sec:asymptotics}

For the sake of completeness, we first recall the definition of Poisson point processes (see, e.g., \citealp{Lieshout2000}; \citealp{Daley2009CampbellMecke}).

\begin{definition}
A spatio-temporal point process $X$ defined as in the beginning of Section \ref{sec:background} is a Poisson point process if it satisfies the following properties: 
\begin{itemize}[itemsep=0.7pt]
    \item[(i)] for any bounded Borel set $B\subset W\times T$, the number of points that fall in $B$ is Poisson distributed with mean $\Lambda(B)$;
    \item[(ii)] in disjoint bounded Borel sets $B_{1}, B_{2} \subset W\times T$, the numbers of points that fall in $B_{1}, B_{2}$ are independent.
\end{itemize}
\end{definition}

The probability distribution of a Poisson point process $X$ is
completely specified by its
first-order moment measure $\Lambda$. Thus, the assumption of a log-linear intensity function parameterized as (\ref{e:intensity}) is quite natural and constitutes an exponential family with the components of $\sum_{x \in X} \bm{z}(x)$ as sufficient statistics. Moreover, the pair correlation function of $X$ is always equal to one. The infill asymptotic regime (cf., Section \ref{subsec:infillasymptotics}) is 
then particularly appropriate for a Poisson point process model because of the conditional independence of the points given a fixed number.

In this section, we establish strong consistency and asymptotic normality for logistic regression estimators in the case of spatio-temporal Poisson point processes. 
For ease of referencing, we list here the conditions ($\mathcal{C}.1$)--($\mathcal{C}.8$) that are required to derive the asymptotic results:

\vspace{-\topsep}
\begin{description}[itemsep=0.1pt]
\item[$(\mathcal{C}.1)$]  $\{Y_i\}$ and $\{E_i\}$ with $i \in \mathbb{N}^{+}$ are two independent sequences of independent and identically distributed 
spatio-temporal point processes on some bounded open set $W\times T \subset \mathbb{R}^2 \times \mathbb{R}$ and defined on the same underlying probability space $\Omega$.  Set $X_n = \cup_{i=1}^n Y_i$ and $D_n = \cup_{i=1}^n E_i$.

\item[$(\mathcal{C}.2)$] $Y_i$ is a Poisson point 
process with intensity function $\lambda(u; \boldsymbol\theta)$ as (\ref{e:intensity}), where $b>0$ is an absolutely integrable function, $\bm{z}$ is a measurable vector of
covariates and the parameter vector $\boldsymbol \theta$ lies in an open set $\boldsymbol\Theta \subset \mathbb{R}^m$; $E_i$ has absolutely integrable intensity function $\rho(u) > 0$. 

\item[$(\mathcal{C}.3)$] $E_{i}$ has bounded pair correlation function $g$, that is,
$\sup_{(u,v)\in (W\times T)^{2}} g(u,v)$ $<\infty$.

\item[$(\mathcal{C}.4)$] For every $\boldsymbol\theta \in \boldsymbol\Theta$, there exist $\epsilon_{1}(\boldsymbol \theta), \epsilon_{2}(\boldsymbol \theta) > 0$ such that $\epsilon_{1}(\boldsymbol \theta) < \inf_{u\in W\times T} \rho(u) / \lambda(u; \boldsymbol \theta)$ and $ \sup_{u\in W\times T} \rho(u) / \lambda(u; \boldsymbol \theta) < \epsilon_{2}(\boldsymbol \theta)$.

\item[$(\mathcal{C}.5)$] The elements of the measurable covariate vector $\bm{z}$ are bounded, that is,
$\sup_{u\in W\times T}\|\bm{z}(u)\| $ $<\infty$.

\item[$(\mathcal{C}.6)$] The parameter space $\boldsymbol\Theta$ is convex.

\item[$(\mathcal{C}.7)$] The parametric model for $\lambda$ is identifiable, that is, $\lambda(u; \boldsymbol \theta) = \lambda(u; \boldsymbol{\tilde{\theta}})$ almost everywhere on $W\times T$ implies $\boldsymbol\theta = \boldsymbol{\tilde{\theta}}$.

\item[$(\mathcal{C}.8)$] The matrix 
$\bm{U}$, whose ($k,l$)-th entry reads
$\int_{W\times T}
\lambda(u;\boldsymbol{\theta_{0}}) \rho(u)  z^{(k)}(u) z^{(l)}(u)/
(\lambda(u;\boldsymbol{\theta_{0}}) + \rho(u))
   du$ with $\boldsymbol{\theta_{0}} \in \boldsymbol\Theta$ and
$1\leq k,l \leq m$, is positive definite.
\end{description}
\vspace{-\topsep}

A few remarks on some of the conditions are appropriate. In condition ($\mathcal{C}.2$), $Y_{i}$ is
assumed to be a Poisson point process, however, it may be relaxed (cf.,\ Section~\ref{subsec:generalpointprocess}).
In condition ($\mathcal{C}.3$), $E_{i}$ is not required to be a Poisson point process, as its realisations are
only used to approximate an integral. 
Condition ($\mathcal{C}.4$) is reasonable, recalling the rule of thumb recommended by \citet{Baddeley2014LRPP} for selecting $\rho(u)$ that $\rho=4\lambda$.
Condition ($\mathcal{C}.7$) is necessary for strong consistency (cf.,\ Theorem~\ref{theorem:consistency}) and condition ($\mathcal{C}.8$) helps ensure the attainment of the logistic regression estimator $\widehat{\boldsymbol{\theta}_{n}}$ as $n \to \infty$ (cf.,\ Theorem~\ref{theorem:asymptotics}).

In the remainder of this paper, we will use $P_{\boldsymbol{\theta_0}}$ to denote the distribution of $(X_n, D_n)$ under the true parameter value $\boldsymbol{\theta_0}$.

\subsection{Strong consistency}
\label{sec:consistency}

To establish strong consistency, we start our investigations with the asymptotic behaviour of the scaled log-likelihood function. 

\begin{lemma}
\label{lemma:4_1}
Assume that the conditions
($\mathcal{C}.1$)--($\mathcal{C}.2$) and
($\mathcal{C}.4$)--($\mathcal{C}.5$) hold. Define 
$l_n(\boldsymbol\theta) = l(X_n, D_n; \boldsymbol \theta)$ by (\ref{e:LR_likelihood})
with $\boldsymbol\theta\in \boldsymbol\Theta$. 
Then, as $n\to\infty$,  $l_{n}(\boldsymbol\theta)/n$
converges $P_{\boldsymbol{\theta_0}}$-almost surely
to
\begin{equation*}
 \int_{W\times T}
 \left\{\lambda(u;\boldsymbol{\theta_{0}})
    \log\left[
       \frac{\lambda(u;\boldsymbol\theta)}
       {\lambda(u;\boldsymbol\theta)+\rho(u)}
    \right]
    + \rho(u)
    \log\left[
       \frac{\rho(u)}
       {\lambda(u;\boldsymbol\theta)+\rho(u)}
    \right]
\right\}du.
\end{equation*}
\end{lemma}

\begin{proof}
Under conditions ($\mathcal{C}.1$)--($\mathcal{C}.2$) and recalling the logistic log-likelihood function (\ref{e:LR_likelihood}),
\begin{equation*}
    \frac{l_{n}(\boldsymbol\theta)}{n} = 
       \frac{1}{n} \sum_{x\in X_{n}} \log\left[
          \frac{\lambda(x;\boldsymbol\theta)}{\lambda(x;\boldsymbol\theta)+\rho(x)}
       \right]+
       \frac{1}{n} \sum_{x\in D_{n}} \log\left[
          \frac{\rho(x)}{\lambda(x;\boldsymbol\theta)+\rho(x)}
       \right],
\end{equation*}
which consists of two first-order $U$-statistics
defined on $X_{n}$ and $D_{n}$. We shall derive 
their strong convergence separately.

Write the first-order $U$-statistic defined on $X_n$ as
\begin{equation*}
    \frac{1}{n} \sum_{x\in X_{n}} \log\left[
       \frac{\lambda(x;\boldsymbol\theta)}
       {\lambda(x;\boldsymbol\theta)+\rho(x)}
    \right]
    = \frac{1}{n} \sum_{i=1}^{n} 
    \left\{ 
    \sum_{x\in Y_i} \log\left[ 
       \frac{\lambda(x;\boldsymbol\theta)}
       {\lambda(x;\boldsymbol\theta)+\rho(x)}
       \right] \right\}.
\end{equation*}
The sum is the average of independent and identically
distributed real-valued random variables. By the
Campbell--Mecke theorem,
\begin{equation}
\mathbb{E}_{\boldsymbol{\theta_0}} \left\{
\sum_{x\in Y_i} \log\left[ 
 \frac{\lambda(x;\boldsymbol\theta)}
 {\lambda(x;\boldsymbol\theta)+\rho(x)} \right] 
 \right\} =
 \int_{W\times T}
 \log\left[ 
  \frac{\lambda(u;\boldsymbol\theta)}
  {\lambda(u;\boldsymbol\theta)+\rho(u)}
  \right]
  \lambda(u;\boldsymbol{\theta_0}) \, du.
 \label{e:SL-mean}
 \end{equation}
Conditions ($\mathcal{C}.2$) and ($\mathcal{C}.5$) imply that the intensity function 
$\lambda( u;\boldsymbol\theta_{0})$ is
absolutely integrable 
on $W\times T$. Furthermore, by conditions
($\mathcal{C}.2$) and ($\mathcal{C}.4$),
\begin{equation*}
\log\left[ \frac{1}{1+\epsilon_{2}(\boldsymbol\theta)} \right]  <  \log\left[ 
       \frac{1}{ 1 + \rho(u) / \lambda(u;\boldsymbol\theta)} \right] 
       =
        \log\left[ 
       \frac{\lambda(u;\boldsymbol\theta)}
       {\lambda(u;\boldsymbol\theta)+\rho(u)} \right] < 
       \log\left[ \frac{1}{1+\epsilon_{1}(\boldsymbol\theta)} \right].
\end{equation*}
Thus, the $P_{\boldsymbol{\theta_0}}$-mean in
(\ref{e:SL-mean}) is finite for all
$\boldsymbol{\theta}$. Then, Kolmogorov's
strong law of large numbers implies that the
first-order $U$-statistic defined on $X_n$ converges
$P_{\boldsymbol{\theta_0}}$-almost surely to the integral in the right-hand side of equation
(\ref{e:SL-mean}).

Similarly, for the first-order $U$-statistic defined
on the dummy point process $D_n$, as $n\to\infty$,
\begin{equation*}
\begin{split}
   \frac{1}{n} \sum_{x\in D_{n}} \log\left[
      \frac{\rho(x)}{\lambda(x;\boldsymbol\theta)+\rho(x)}
   \right]
 & \stackrel{P_{\boldsymbol{\theta_0}}-a.s.}{\to} \mathbb{E}_{\boldsymbol{\theta_0}}\left\{
    \sum_{x\in E_i}\log\left[
        \frac{\rho(x)}{\lambda(x;\boldsymbol\theta)+\rho(x)}
    \right] \right\} \\
  & = \int_{W\times T} \log\left[
     \frac{\rho(u)}{\lambda(u;\boldsymbol\theta)+\rho(u)}
  \right]\rho(u)du
\end{split}
\end{equation*}
because $\rho(u)$ is absolutely integrable by condition ($\mathcal{C}.2$) and 
\begin{equation*}
\log\left[ 
 \frac{\epsilon_{1}(\boldsymbol\theta)}
 {1+\epsilon_{1}(\boldsymbol\theta)} 
 \right]  <  \log\left[ 
       \frac{\rho(u) / \lambda(u;\boldsymbol\theta)}
         { 1+\rho(u) / \lambda(u;\boldsymbol\theta)} \right]
 = \log\left[ 
       \frac{\rho(u)}{\lambda(u;\boldsymbol\theta)+\rho(u)} \right] < \log\left[ 
 \frac{\epsilon_{2}(\boldsymbol\theta)}
 {1+\epsilon_{2}(\boldsymbol\theta)} 
 \right] 
\end{equation*}
by conditions ($\mathcal{C}.2$) and ($\mathcal{C}.4$). 

The proof is then completed by combining these two
strong convergence results.
\end{proof}

Recall that the logistic regression estimator 
minimizes 
$U_n(\boldsymbol\theta) =
-l_{n}(\boldsymbol\theta)/{n}$ 
and thus can be considered as a minimum contrast estimator.
The next theorem is then concerned with strong consistency.
%

\begin{theorem}
\label{theorem:consistency}
Assume that the conditions
($\mathcal{C}.1$)--($\mathcal{C}.2$) and
($\mathcal{C}.4$)--($\mathcal{C}.7$) hold. Define 
$l_n(\boldsymbol\theta) =
 l(X_n, D_n; \boldsymbol \theta)$ by
(\ref{e:LR_likelihood}) with 
$\boldsymbol\theta\in \boldsymbol\Theta$ and set
$\widehat{\boldsymbol{\theta}_{n}}
= \argmax_{\boldsymbol\theta \in \boldsymbol{\Theta}}
l_{n}(\boldsymbol\theta)$. If $\widehat{\boldsymbol{\theta}_{n}}$ is attained,
then
as $n\rightarrow \infty$,
$\widehat{\boldsymbol{\theta}_{n}}$ converges
$P_{\boldsymbol{\theta_0}}$-almost surely to
$\boldsymbol{\theta_0}$.
\end{theorem}

\begin{proof}
Conditions ($\mathcal{C}.2$) and
($\mathcal{C}.6$) ensure that the parameter space
$\boldsymbol\Theta$ is open and convex. 

Firstly, we prove that, for every $\omega$ in the
underlying probability space $\Omega$, the realisations of the
function
$\boldsymbol\theta \mapsto U_{n}(\boldsymbol\theta)$
are convex. By condition ($\mathcal{C}.2$), the
intensity function $\lambda(x;\boldsymbol\theta)$ 
has the log-linear form (\ref{e:intensity})
and is thus twice differentiable with respect to the
parameter vector $\boldsymbol\theta \in \boldsymbol\Theta$.
Then, the Hessian matrix of
$-l_{n}(\boldsymbol\theta)$ reads
\begin{equation*}
    \left[ \sum_{x\in X_{n} \cup D_{n}} 
        \frac{\lambda(x;\boldsymbol\theta)\rho(x)}
        {\left(\lambda(x;\boldsymbol\theta) + \rho(x)\right)^{2}}
        z^{(k)}(x) z^{(l)}(x)
     \right]_{k,l=1}^{m}.
\end{equation*}
By decomposition, it can be written into the product of a matrix $\bm{M}$ and its transpose as $\bm{M}^{\top}\bm{M}$, where $\bm{M}$ is a [$|X_{n}\cup D_{n}|\times m$]-matrix given by
\begin{equation*}
   \bm{M} = \left[
       \frac{(\lambda(x_{i};\boldsymbol\theta) \rho(x_{i}))^{1/2}}
         {\lambda(x_{i};\boldsymbol\theta) + \rho(x_{i})}
     z^{(k)}(x_{i})
   \right]_{i=1,k=1}^{\left|X_{n}\cup D_{n}\right|,m}.
\end{equation*}
Here, $|X_{n}\cup D_{n}|$ denotes the number of
points in $X_{n}\cup D_{n}$, the subscript $i$ runs
through all points in $X_{n}\cup D_{n}$ and $k$ runs
through the $m$ covariates. One can readily
obtain that the Hessian matrix of
$-l_{n}(\boldsymbol\theta)$ is positive semi-definite for every $\boldsymbol\theta \in \boldsymbol\Theta$,
which implies that the function
$\boldsymbol\theta \mapsto -l_{n}(\boldsymbol\theta)$, and thus
$-l_n(\boldsymbol\theta) / n$, is convex.

Secondly, we prove that, as $n\rightarrow\infty$, 
$U_{n}(\boldsymbol\theta) - U_{n}(\boldsymbol{\theta_0})$ converges
$P_{\boldsymbol{\theta_0}}$-almost surely to a
function $K(\boldsymbol\theta,\boldsymbol{\theta_0})$
that is non-negative and vanishes only at
$\boldsymbol\theta = \boldsymbol{\theta_0}$. 
By Lemma \ref{lemma:4_1}, under
conditions ($\mathcal{C}.1$)--($\mathcal{C}.2$) 
and ($\mathcal{C}.4$)--($\mathcal{C}.5$),
$U_n(\boldsymbol\theta) - U_n(\boldsymbol{\theta_0})$
converges $P_{\boldsymbol{\theta_0}}$-almost surely to
\begin{equation*}
\int_{W\times T} \left\{ \lambda(u;\boldsymbol{\theta_0})
    \log\left[  
          \frac{\lambda(u;\boldsymbol{\theta_0})}
          {\lambda(u;\boldsymbol\theta)}
    \right]
- \left(  \lambda(u;\boldsymbol{\theta_0}) + \rho(u) \right)
  \log \left[
         \frac{\lambda(u;\boldsymbol{\theta_0}) + \rho(u)}
         {\lambda(u;\boldsymbol\theta) + \rho(u)}
   \right]
\right\} du.
\end{equation*}
Denote this limit by
$K(\boldsymbol\theta,\boldsymbol{\theta_0})$ 
and the integrand by
$k(u;\boldsymbol\theta,\boldsymbol{\theta_0})$. 
Then, clearly 
$K(\boldsymbol{\theta_0}, \boldsymbol{\theta_0})
= 0$. Furthermore,
\begin{equation*}
 k(u;\boldsymbol\theta,\boldsymbol{\theta_0}) = 
     \lambda(u;\boldsymbol\theta) \left\{
        \frac{\lambda(u;\boldsymbol{\theta_0})}
             {\lambda(u;\boldsymbol\theta)}
  \log\left[ 
     \frac{\lambda(u;\boldsymbol{\theta_0})}
          {\lambda(u;\boldsymbol\theta)} 
  \right]
- \left(\frac{\lambda(u;\boldsymbol{\theta_0})}
             {\lambda(u;\boldsymbol\theta)} + 
   \frac{\rho(u)}{\lambda(u;\boldsymbol\theta)}\right)
  \log\left[
     \frac{\frac{\lambda(u;\boldsymbol{\theta_0})}
          {\lambda(u;\boldsymbol\theta)}
          + \frac{\rho(u)}{\lambda(u;\boldsymbol\theta)}}
   {1+\frac{\rho(u)}{\lambda(u;\boldsymbol\theta)}}
  \right]
\right\}.
\end{equation*}
By condition ($\mathcal{C}.2$), 
$\lambda(u;\boldsymbol\theta)>0$ as $b(u)>0$.
Consider the function 
$a\mapsto a\log a-(a+b)\log[(a+b)/(1+b)]$ 
with $a,b>0$. Its derivative with respect to $a$ is
$\log[a(1+b)/(a+b)]$. Hence, the function is strictly
decreasing when $a\in(0,1)$ and strictly increasing
when $a\in(1,+\infty)$. For $a=1$, 
$a\log a-(a+b)\log[(a+b)/(1+b)] = 0$. Thus,
$k(u;\boldsymbol\theta,\boldsymbol{\theta_0})$ is
non-negative, and is strictly positive when
$\lambda(u;\boldsymbol \theta) \neq
\lambda(u;\boldsymbol{\theta_0})$.

Under condition ($\mathcal{C}.7$), 
strong consistency follows from an appeal to the Proposition below
\citet[Theorem~3.4.4]{Guyon1995consistency}.
\end{proof}

\begin{remark}
It is important to note that the proofs of Lemma~\ref{lemma:4_1} and Theorem~\ref{theorem:consistency} do not depend on 
the assumption that $Y_{i}$, and thus $X_n$, is a Poisson point process. However, strong consistency relies on the concavity of the log-likelihood function of logistic regression estimation.
\end{remark}

\subsection{Asymptotic normality}
\label{sec:asymptotic_normality}

To establish asymptotic normality, we start our investigations
with the Taylor series of the score function
(\ref{e:LR_score_n}).

For every component of the score function, denoted by $\boldsymbol{s}_{n}^{(i)}(\boldsymbol\theta)$ with $1 \leq i \leq m$, the second-order Taylor expansion of $\boldsymbol{s}_{n}^{(i)}(\widehat{\boldsymbol{\theta}_{n}})$ with respect to $\boldsymbol\theta$ at $\boldsymbol{\theta_{0}}$ reads
\begin{equation}
 0 = \boldsymbol{s}_{n}^{(i)}(\widehat{\boldsymbol{\theta}_{n}})
    = \boldsymbol{s}_{n}^{(i)}(\boldsymbol{\theta_0}) 
    + 
    \nabla \boldsymbol{s}_{n}^{(i)}(\boldsymbol{\theta_0})
    ( \widehat{\boldsymbol{\theta}_{n}} - \boldsymbol{\theta_0} )
    + 
    \frac{1}{2} 
    ( \widehat{\boldsymbol{\theta}_{n}} - \boldsymbol{\theta_0} )^{\top}
    \nabla^{2} \boldsymbol{s}_{n}^{(i)}(\boldsymbol{\theta^{\prime, i}})
    ( \widehat{\boldsymbol{\theta}_{n}} - \boldsymbol{\theta_0} ),
    \label{e:Taylor_series_error}
\end{equation}
where $\nabla \boldsymbol{s}_{n}^{(i)}(\boldsymbol{\theta})$ is
the $[1\times m]$-vector containing the first-order
partial derivatives of $\boldsymbol{s}_{n}^{(i)}(\boldsymbol{\theta})$ with respect to $\boldsymbol\theta$ and 
$\nabla^{2} \boldsymbol{s}_{n}^{(i)}(\boldsymbol{\theta})$ is
the $[m\times m]$-matrix containing the 
second-order partial derivatives of
$\boldsymbol{s}_{n}^{(i)}(\boldsymbol{\theta})$. Moreover,
$\boldsymbol{\theta^{\prime,i}}$ is a convex combination of $\widehat{\boldsymbol{\theta}_{n}}$ and
$\boldsymbol{\theta_0}$ which, by condition
($\mathcal{C}.6$), lies in $\boldsymbol\Theta$ as well. 

Write $\nabla\boldsymbol s_n(\boldsymbol{\theta_{0}})$ for the matrix whose $i$-th row is $\boldsymbol{s}_{n}^{(i)}(\boldsymbol{\theta_0})$ with
$1 \leq i \leq m$ and assume that its inverse is
well-defined. Heuristically, 
the idea is then to ignore the error term
which is the quadratic form in
(\ref{e:Taylor_series_error}) and rearrange the
remaining terms to obtain
\begin{equation}
   n^{1/2}(\widehat{\boldsymbol\theta_{n}} - \boldsymbol{\theta_0})
   \approx \left[-\frac{\nabla \boldsymbol{s}_{n}(\boldsymbol{\theta_0})}{n}\right]^{-1}
   \frac{\boldsymbol{s}_{n}(\boldsymbol{\theta_0})}{n^{1/2}}.
    \label{e:Taylor_series_round}
\end{equation} 
In the next two lemmas, we first study the asymptotic
behaviours of the terms 
$- \nabla \boldsymbol{s}_{n}(\boldsymbol{\theta_0})/n$ and
$\boldsymbol{s}_{n}(\boldsymbol{\theta_0})/n^{1/2}$, respectively.

\begin{lemma}
\label{lemma:4_3}
Assume that the conditions
($\mathcal{C}.1$)--($\mathcal{C}.2$) and
($\mathcal{C}.5$) hold. Define
$\boldsymbol{s}_{n}(\boldsymbol\theta)$ 
$=$ $ \boldsymbol{s}(X_n, D_n;$ $ \boldsymbol\theta)$ by
(\ref{e:LR_score_n}) with $\boldsymbol\theta \in
\boldsymbol\Theta$. Then, as 
$n\rightarrow \infty$, 
$-\nabla \boldsymbol{s}_{n}(\boldsymbol{\theta_0})/n$ converges
$P_{\boldsymbol{\theta_0}}$-almost surely to the
matrix $\bm{U}$ given by
\begin{equation*}
\bm{U}=\left[
   \int_{W\times T}
      \frac{\lambda(u;\boldsymbol{\theta_0}) \rho(u)}
           {\lambda(u;\boldsymbol{\theta_0}) + \rho(u)} 
     z^{(k)}(u) z^{(l)}(u) du
\right]_{k,l=1}^{m}.
\end{equation*}
\end{lemma}

\begin{proof}
Under conditions ($\mathcal{C}.1$)--($\mathcal{C}.2$) and recalling the score function (\ref{e:LR_score_n}), 
\begin{equation*}
-\frac{ \nabla \boldsymbol{s}_{n}(\boldsymbol{\theta_0}) }{ n } = 
   \left[
       \sum_{x\in X_{n}\cup D_{n}} 
           \frac{\lambda(x;\boldsymbol{\theta_0})\rho(x)}
         {n \left(
           \lambda(x;\boldsymbol{\theta_0}) + \rho(x)\right)^{2}}
      z^{(k)}(x) z^{(l)}(x)
   \right]_{k,l=1}^{m}.
\end{equation*}
We shall prove component-wise strong convergence. 

Consider the $(k,l)$-th entry of the matrix above which, recalling condition ($\mathcal{C}.1$), is given by
\begin{equation*}
   \frac{1}{n} \sum_{i=1}^n \left\{\sum_{x\in Y_{i}\cup E_{i}} 
           \frac{\lambda(x;\boldsymbol{\theta_0})\rho(x)}
         { \left(
           \lambda(x;\boldsymbol{\theta_0}) + \rho(x)\right)^{2}}
      z^{(k)}(x) z^{(l)}(x) \right\} .
\end{equation*}
The sum is the average of independent and identically distributed real-valued random variables. By the Campbell--Mecke theorem,
\begin{equation*}
\mathbb{E}_{\boldsymbol{\theta_0}} \left\{
   \sum_{x\in Y_{i}\cup E_{i}} 
       \frac{\lambda(x;\boldsymbol{\theta_0}) \rho(x)}
         {\left(\lambda(x;\boldsymbol{\theta_0}) + \rho(x) \right)^{2}}
         z^{(k)}(x) z^{(l)}(x)
\right\}
   =
 \int_{W\times T}
 \frac{\lambda(u;\boldsymbol{\theta_{0}})\rho(u)}
 {\lambda(u;\boldsymbol{\theta_{0}})+\rho(u)}
   z^{(k)}(u)z^{(l)}(u) du.
\end{equation*}
By condition ($\mathcal{C}.5$), the covariate terms in the integrand above are all bounded. Furthermore, by condition ($\mathcal{C}.2$),
\begin{equation*}
0 < \frac{\lambda(u;\boldsymbol{\theta_0}) \rho(u)}
  {\lambda(u;\boldsymbol{\theta_0}) + \rho(u)} < \rho(u)
\end{equation*}
and $\rho(u)$ is absolutely integrable.  Kolmogorov's strong
law of large numbers implies the claimed 
$P_{\boldsymbol{\theta_0}}$-almost sure convergence.
\end{proof}

\begin{lemma}
\label{lemma:4_4}
Assume that the conditions
($\mathcal{C}.1$)--($\mathcal{C}.3$) and
($\mathcal{C}.5$) hold. Define
$\boldsymbol{s}_{n}(\boldsymbol\theta) = \boldsymbol{s}(X_n, D_n;$
$\boldsymbol\theta)$ by (\ref{e:LR_score_n}) with
$\boldsymbol\theta \in \boldsymbol\Theta$. Then, as
$n\rightarrow \infty$,
$\boldsymbol{s}_{n}(\boldsymbol{\theta_0})/n^{1/2}$ converges 
under $P_{\boldsymbol{\theta_0}}$ in distribution to an
$m$-dimensional normally distributed random vector
with mean zero and covariance matrix
\begin{equation*}
\begin{split}
    \bm{V} &= \left[
      \int_{W\times T}
          \frac{\lambda(u;\boldsymbol{\theta_0}) \rho(u)}
               {\lambda(u;\boldsymbol{\theta_0}) + \rho(u)}
      z^{(k)}(u) z^{(l)}(u) du
\right]_{k,l=1}^{m}\\
&+
\left[
     \int_{W\times T}\int_{W\times T}
     \frac{
       \lambda(u;\boldsymbol{\theta_0})
       \lambda(v;\boldsymbol{\theta_0})
       \rho(u)\rho(v) z^{(k)}(u) z^{(l)}(v) }
     {(\lambda(u;\boldsymbol{\theta_0}) + \rho(u))
     (\lambda(v;\boldsymbol{\theta_0}) + \rho(v))} \left(g(u,v)-1\right)
     du dv \right]_{k,l=1}^{m}.
\end{split}
\end{equation*}
\end{lemma}

\begin{proof}
Under conditions ($\mathcal{C}.1$)--($\mathcal{C}.2$) and recalling the score function (\ref{e:LR_score_n}), 
\begin{equation*}
    \frac{\boldsymbol{s}_{n}(\boldsymbol{\theta_0})}{n^{1/2}} =
    \sum_{x\in X_{n}} 
    \frac{\rho(x)}
    { n^{1/2} \left( 
        \lambda(x;\boldsymbol{\theta_0}) + \rho(x)
     \right)}
    \boldsymbol{z}(x)  
    - \sum_{x\in D_{n}}
    \frac{\lambda(x;\boldsymbol{\theta_0})  }{n^{1/2} \left(
    \lambda(x;\boldsymbol{\theta_0}) + \rho(x)\right)}
    \boldsymbol z(x).
\end{equation*}
It consists of two first-order $U$-statistics defined on $X_n$ and $D_n$. We write it as
\begin{equation*}
\begin{split}
    \frac{\boldsymbol{s}_{n}( \boldsymbol{\theta_0}) }{n^{1/2}}
    &= \frac{\boldsymbol{s}_{n}(X_n; \boldsymbol{\theta_0})}{n^{1/2}} - \frac{\boldsymbol{s}_{n}(D_n; \boldsymbol{\theta_0})}{n^{1/2}}\\
     &= n^{1/2}\left\{\frac{1}{n}\sum_{i=1}^{n}\sum_{x\in Y_{i}}
    \frac{\rho(x)}
    {\lambda(x;\boldsymbol{\theta_0}) + \rho(x)}
    \boldsymbol{z}(x)  
    - \frac{1}{n}\sum_{i=1}^{n}\sum_{x\in E_{i}}
    \frac{\lambda(x;\boldsymbol{\theta_0})  }{
    \lambda(x;\boldsymbol{\theta_0}) + \rho(x)}
    \boldsymbol z(x)\right\}
\end{split}
\end{equation*}
and discuss the two first-order $U$-statistics separately.

Consider $n^{-1/2} \boldsymbol{s}_{n}(X_n; \boldsymbol{\theta_0})$.
Note that the first term between the curly brackets
above is again the average of independent and identically
distributed real-valued random vectors. By the
Campbell--Mecke theorem and recalling Section \ref{subsec:momentmeasures},
\begin{equation*}
    \text{\bf{M}}_{\boldsymbol{\theta_0}} = \mathbb{E}_{\boldsymbol{\theta_0}}
    \left\{
       \sum_{x\in Y_{i}}
       \frac{\rho(x) \boldsymbol z(x) }
              {\lambda(x;\boldsymbol{\theta_0})+\rho(x)}
    \right\}
    = \left[\int_{W\times T}
      \frac{\rho(u) \lambda(u;\boldsymbol{\theta_0}) z^{(l)}(u) }
            {\lambda(u;\boldsymbol{\theta_0}) + \rho(u)} 
           du\right]_{l=1}^{m}
\end{equation*}
and, because $Y_i$ is a Poisson
point process under condition ($\mathcal{C}.2$),
\begin{equation}
    \Cov_{\boldsymbol{\theta_0}}
     \left
        \{\sum_{x\in Y_{i}} 
        \frac{\rho(x) \boldsymbol z(x) }
        {\lambda(x;\boldsymbol{\theta_0}) + \rho(x)}
    \right\}
    =
   \left[
     \int_{W\times T}
     \frac{\lambda(u;\boldsymbol{\theta_0})\rho^{2}(u) z^{(k)}(u) z^{(l)}(u) }
     {(\lambda(u;\boldsymbol{\theta_0}) + \rho(u) )^{2}} 
     du \right]_{k,l=1}^{m}.
     \label{e:V}
\end{equation}
By condition ($\mathcal{C}.5$), the covariate terms
in the integrand above are all bounded. Furthermore,
by condition ($\mathcal{C}.2$),
\begin{equation*}
  0 < \frac{\lambda(u;\boldsymbol{\theta_0}) \rho^{2}(u)}
  {(\lambda(u;\boldsymbol{\theta_0}) + \rho(u))^{2}} < \lambda(u;\boldsymbol{\theta_0}), \quad 
  0 < \frac{\lambda(u;\boldsymbol{\theta_0})\lambda(v;\boldsymbol{\theta_0})\rho(u)\rho(v)}
  {(\lambda(u;\boldsymbol{\theta_0}) + \rho(u))(\lambda(v;\boldsymbol{\theta_0}) + \rho(v))} < \lambda(u;\boldsymbol{\theta_0})\lambda(v;\boldsymbol{\theta_0})
\end{equation*}
and $\lambda(u;\boldsymbol{\theta_0}), \lambda(v;\boldsymbol{\theta_0})$ are absolutely integrable. Then, the multi-variate Lindeberg–L\'evy central limit
theorem implies that
$n^{-1/2} ( \boldsymbol{s}_{n}(X_n; \boldsymbol{\theta_0})
- n \text{\bf{M}}_{\boldsymbol{\theta_0}} ) $ converges under
$P_{\boldsymbol{\theta_0}}$ in distribution to an
$m$-dimensional normally distributed random vector 
with mean zero and covariance matrix given by (\ref{e:V}).

Analogously, by condition ($\mathcal{C}.3$),
$n^{-1/2} ( \boldsymbol{s}_{n}(D_n, \boldsymbol{\theta_0})
- n \text{\bf{M}}_{\boldsymbol{\theta_0}} )$
converges under $P_{\boldsymbol{\theta_0}}$ in
distribution to an $m$-dimensional normally
distributed random vector with mean zero
and covariance matrix given by
\begin{equation*}
\begin{split}
 &\Cov_{\boldsymbol{\theta_0}}\left\{
    \sum_{x\in E_{i}}
     \frac{ \lambda(x;\boldsymbol{\theta_0}) \boldsymbol z(x) }
        {\lambda(x;\boldsymbol{\theta_0}) + \rho(x)}
    \right\}    
    =
    \left[
       \int_{W\times T}
      \frac{ \lambda^{2}(u;\boldsymbol{\theta_0}) \rho(u) z^{(k)}(u) z^{(l)}(u) }
       {(\lambda(u;\boldsymbol{\theta_0}) + \rho(u))^{2}}
      du \right]_{k,l=1}^{m}\\
 &+ \left[
     \int_{W\times T}\int_{W\times T}
     \frac{\lambda(u;\boldsymbol{\theta_0})\lambda(v;\boldsymbol{\theta_0})\rho(u)\rho(v) z^{(k)}(u) z^{(l)}(v) }
     {(\lambda(u;\boldsymbol{\theta_0}) + \rho(u))(\lambda(v;\boldsymbol{\theta_0}) + \rho(v))} (g(u,v)-1)
     dudv \right]_{k,l=1}^{m}.  
\end{split}
\end{equation*}

Applying L\'evy's continuity theorem and using the independence of $X_n$ and $D_n$ under condition ($\mathcal{C}.1$), the 
weak limits for the two first-order $U$-statistics can be combined, which completes the proof.
\end{proof}

\begin{remark}
The pair correlation function $g(u,v)$ of $E_{i}$ can be tuned to control the covariance matrix $\bm{V}$.
When $E_{i}$, and thus $D_n$, is a Poisson point process, $g(u,v)\equiv1$
and the second term in $\bm{V}$
vanishes.  Moreover, regular point processes of $E_{i}$, whose $g(u,v)<1$, may be useful to reduce the variance \citep{Gautier2019DPP}. 
\end{remark}

Recalling (\ref{e:Taylor_series_round}), we are
now in a position to conjecture a central limit
theorem for the logistic regression estimator. For a
formal proof, though, we need to analyze the error
term in the Taylor series
(\ref{e:Taylor_series_error}).
The next theorem is then concerned with asymptotic normality.

\begin{theorem}
\label{lemma:4_5}
Assume that the conditions ($\mathcal{C}.1$)--($\mathcal{C}.8$) hold. Define  $\boldsymbol{s}_{n}(\boldsymbol\theta)
= \boldsymbol{s}(X_n, D_n; \boldsymbol\theta)$ by (\ref{e:LR_score_n}) with
$\boldsymbol\theta \in \boldsymbol\Theta$ and
let $\widehat{\boldsymbol{\theta}_{n}}$ be the logistic
regression estimator for which $\boldsymbol s_n(
\widehat{\boldsymbol{ \theta}_n}) = \bm0$.
If $\widehat{
\boldsymbol{\theta}_{n}}$ is attained, then as $n\to\infty$, 
$n^{1/2}
(\widehat{\boldsymbol{\theta}_{n}}-
\boldsymbol{\theta_0})$ converges under 
$P_{\boldsymbol{\theta_0}}$ in
distribution to an
$m$-dimensional normally distributed random
vector with mean zero and 
covariance matrix $\bm{U}^{-1}\bm{V}(\bm{U}^{-1})^{\top}$, 
where $\bm{U}, \bm{V}$ are as defined in 
Lemma~\ref{lemma:4_3} and Lemma~\ref{lemma:4_4}.
\end{theorem}

\begin{proof}
Consider the error term in the Taylor series
(\ref{e:Taylor_series_error}) which,
under conditions ($\mathcal{C}.1$) and
($\mathcal{C}.2$) and recalling the score function (\ref{e:LR_score_n}), reads
\begin{equation}
\begin{split}
    & \frac{1}{2}( \widehat{\boldsymbol{\theta}_{n}} - \boldsymbol{\theta_0} )^{\top}
    \nabla^{2} \boldsymbol{s}_{n}^{(i)}(\boldsymbol{\theta^{\prime,i}})
    ( \widehat{\boldsymbol{\theta}_{n}} - \boldsymbol{\theta_0} )
    = 
    \frac{1}{2}\sum_{k=1}^m \sum_{l=1}^m 
    (\widehat{\boldsymbol{\theta}_{n}} - \boldsymbol{\theta_0}) ^{(k)} 
    (\widehat{\boldsymbol{\theta}_{n}} - \boldsymbol{\theta_0}) ^{(l)} \times \\
    & \times
    \sum_{x\in X_n\cup D_n}
    \frac{
    \left( 
     \lambda(x; \boldsymbol{\theta^{\prime,i}}) - \rho(x) \right)
    \lambda(x; \boldsymbol{\theta^{\prime,i}}) \rho(x) z^{(i)}(x) z^{(k)}(x) z^{(l)}(x)
    }{
     \left( \lambda(x; \boldsymbol{\theta^{\prime,i}}) + \rho(x) \right)^3
    }
\end{split}
\label{e:lemma_4_5_3}
\end{equation}
 for all choices of $\boldsymbol{\theta^{\prime,i}}$ among the convex combinations of $\widehat{\boldsymbol{\theta}_{n}}$ and
$\boldsymbol{\theta_0}$.

Firstly, note that 
\begin{equation}
   \sum_{x\in X_n\cup D_n}
\frac{
\left( \lambda(x; \boldsymbol{\theta^{\prime,i}}) - \rho(x) \right)
\lambda(x; \boldsymbol{\theta^{\prime,i}}) \rho(x) z^{(i)}(x) z^{(k)}(x) z^{(l)}(x)
}{
 n \left( \lambda(x; \boldsymbol{\theta^{\prime,i}}) + \rho(x) \right)^3
}
\label{e:lemma_4_5_6}
\end{equation}
is bounded in absolute value by
\begin{equation*}
  G^{(n)}_{i,k,l} =  \sum_{x\in X_{n}\cup D_{n}}\frac{\left|z^{(i)}(x)z^{(k)}(x)z^{(l)}(x)\right|}{n}.
\end{equation*}
This bound does not depend on $\boldsymbol{\theta^{\prime,i}}$ and thus does not depend on $\widehat{\boldsymbol{\theta}_{n}}$. 
By the Campbell--Mecke theorem,
\begin{equation*}
    \mathbb{E}_{\boldsymbol{\theta_0}}\left\{G^{(n)}_{i,k,l}\right\}=\int_{W\times T}(\lambda(u;\boldsymbol{\theta_0})+\rho(u))\left|z^{(i)}(u)z^{(k)}(u)z^{(l)}(u)\right|du.
\end{equation*}
Under conditions
($\mathcal{C.}2$) and ($\mathcal{C.}5$), $\mathbb{E}_{\boldsymbol{\theta_0}}\{G^{(n)}_{i,k,l}\}$ is a non-negative constant. If it is zero, 
the product of the covariate terms inside has to be zero almost everywhere on $W\times T$, which contradicts condition ($\mathcal{C.}8$). Thus, $\mathbb{E}_{\boldsymbol{\theta_0}}\{G^{(n)}_{i,k,l}\}$ is strictly positive. Then, by Markov's inequality, for any
$\delta>0$, there always exists a finite $H_{i,k,l}(\delta)=\mathbb{E}_{\boldsymbol{\theta_0}}\{G^{(n)}_{i,k,l}\}/\delta$ such that
\begin{equation*}
    \mathbb{P_{\boldsymbol{\theta_0}}}\left\{\left|G^{(n)}_{i,k,l}\right|\geq H_{i,k,l}(\delta)\right\}
    \leq 
    \delta.
\end{equation*}
Since $H_{i,k,l}(\delta)$ depends only on $\delta$ but not on $n$, $G^{(n)}_{i,k,l}$, and thus (\ref{e:lemma_4_5_6}), converges under $P_{\boldsymbol{\theta_0}}$ in probability to $O(1)$. In addition, note that this result applies to every component of (\ref{e:lemma_4_5_6}) with $1 \leq i,k,l \leq m$.

Now, move back to the Taylor expansion (\ref{e:Taylor_series_error}). By Lemma~\ref{lemma:4_3},
$-\nabla \boldsymbol{s}_n(\boldsymbol{\theta_0})/n$ converges
almost surely, and thus in probability, under $P_{\boldsymbol{\theta_0}}$ to 
$\boldsymbol{U}$.
Then, collecting all $i$ with $1 \leq i \leq m$ and recalling the results obtained above, (\ref{e:Taylor_series_error}) can be rewritten as
\begin{equation*}
\left[ \boldsymbol{U} + o_P(1) -
\frac{1}{2} (\widehat{\boldsymbol{ \theta}_n} - 
\boldsymbol{\theta_0})^{\top} O_P(1)  \right] (\widehat{ \boldsymbol{\theta}_n } - \boldsymbol{\theta_0})=\frac{ \boldsymbol{s}_n(\boldsymbol{\theta_0})}{n}.
\end{equation*}
Next, we appeal to strong consistency (cf.,\ Theorem~\ref{theorem:consistency}) to conclude
that $(\widehat{\boldsymbol{ \theta}_n} - 
\boldsymbol{\theta_0})^{\top} O_P(1) $ converges under 
$P_{\boldsymbol{\theta_0}}$ in probability to zero.
Furthermore, by condition $\mathcal{C}.8$, the matrix $\boldsymbol{U}$ is invertible. 
Thus, $\boldsymbol{U} + o_P(1)$ is also invertible with a probability tending to one as $n \to \infty$. Multiplication by this inverse then yields that
\begin{equation*}
 n^{1/2} ( \widehat{ \boldsymbol{\theta}_n } - \boldsymbol{\theta_0} ) = 
 \left[\boldsymbol{U} + o_P(1) \right]^{-1} 
 \frac{  \boldsymbol{s}_n(\boldsymbol{\theta_0})}{n^{1/2} }.
\end{equation*}

The remainder of the proof is then a straightforward application of Slutsky’s theorem. Obviously,
$[\boldsymbol{U} + o_P(1) ]^{-1} $
converges under $P_{\boldsymbol{\theta_0}}$ in probability to $\bm{U}^{-1}$. 
By Lemma \ref{lemma:4_4},
$\boldsymbol{s}_{n}(\boldsymbol{\theta_0})/n^{1/2}$ converges under $P_{\boldsymbol{\theta_0}}$ in distribution to an
$m$-dimensional normally distributed random vector
with mean zero and covariance matrix
$\bm{V}$. Thus, as $n\rightarrow\infty$,
$n^{1/2}(
\widehat{\boldsymbol{\theta}_{n}}
-\boldsymbol{\theta_0})$ converges under
$P_{\boldsymbol{\theta_0}}$ in distribution to an
$m$-dimensional normally distributed random vector
with mean zero and covariance matrix
$\bm{U}^{-1}\bm{V}(\bm{U}^{-1})^{\top}$.
\end{proof}

\begin{remark}
\label{remark:4_6}
In Theorem \ref{theorem:consistency}, we have proved that the logistic regression estimator
$\widehat{\boldsymbol{\theta}_{n}}$ is a strongly
consistent estimator. Here, observing that $\bm{U}$ and $\bm{V}$ are constant matrices, we thus further provide the convergence rate, which is $(\widehat{\boldsymbol{\theta}_{n}} - \boldsymbol{\theta_0})^{(l)}=O_P(n^{-1/2})$ for every $1 \leq l \leq m$.
\end{remark}

\begin{remark}
\label{remark:practical_calculation}
The approximate variance 
$n^{-1} \bm{U}^{-1}\bm{V}(\bm{U}^{-1})^{\top}$ of
${\widehat{\boldsymbol{\theta}_{n}}}$ 
decreases, as $n$ increases. The covariance
matrix depends on the unknown intensity 
function $\lambda(u; \boldsymbol{\theta_0})$.
In practice, a plug-in estimator of $\lambda(u; \boldsymbol{\theta_0})$ would be
used and the dummy point process $D_{n}$ can be employed to
approximate the integrals inside based on the Campbell--Mecke theorem. For instance,
\begin{equation*}
\begin{split}
n \bm V & \approx \left[ 
\sum_{x\in D_n} 
\frac{\lambda(x; \widehat{\boldsymbol\theta_{n}})
z^{(k)}(x) z^{(l)}(x)}
{ \lambda(x; \widehat{\boldsymbol\theta_{n}}) 
  + \rho(x)} 
\right]_{k,l=1}^m 
+ 
\left[
\sJM_{x,y\in D_{n}}
 \frac{\lambda(x;\widehat{\boldsymbol\theta_{n}})
   z^{(k)}(x)}
 {\lambda(x;\widehat{\boldsymbol\theta_{n}})
   +\rho(x)}
 \frac{\lambda(y;\widehat{\boldsymbol\theta_{n}})
 z^{(l)}(y)}
 {\lambda(y;\widehat{\boldsymbol\theta_{n}})
 +\rho(y)} \right]_{k,l=1}^{m}  \\
 &-
\left[
\sum_{x\in D_{n}}
\frac{\lambda(x;\widehat{\boldsymbol\theta_{n}})
z^{(k)}(x)}
{\lambda(x;\widehat{\boldsymbol\theta_{n}})
+\rho(x)}
\sum_{y\in D_{n}}
\frac{\lambda(y;\widehat{\boldsymbol\theta_{n}})
z^{(l)}(y)}
{\lambda(y;\widehat{\boldsymbol\theta_{n}})
+\rho(y)}\right]_{k,l=1}^{m}.
\end{split}
\end{equation*}
Similarly, $n \bm U$ can be approximated by the first
of the three terms in the right hand side of
the equation above. When $E_{i}$, and thus $D_{n}$, is a Poisson
point process, $\bm U = \bm V$ and $\bm{U}^{-1}\bm{V}(\bm{U}^{-1})^{\top}$ reduces to $\bm{U}^{-1}$.
The construction of approximate confidence
intervals for the components of $\bm{\theta_0}$
is now straightforward.
\end{remark}

Note that both strong consistency and asymptotic normality of the logistic regression estimator rely on the attainment of $\widehat{\boldsymbol\theta_{n}}$. Hence, for the sake of completeness, we also formulate the asymptotic existence of the logistic
regression estimator in the next theorem.

\begin{theorem}
\label{theorem:asymptotics}
Assume that the conditions ($\mathcal{C}.1$)--($\mathcal{C}.8$) hold. Define  $\boldsymbol{s}_{n}(\boldsymbol\theta)
= \boldsymbol{s}(X_n, D_n; \boldsymbol\theta)$ by (\ref{e:LR_score_n}) with
$\boldsymbol\theta \in \boldsymbol\Theta$.
Then, for every $n$, an estimator $\widehat{\boldsymbol{\theta}_{n}}$ exists that solves $\boldsymbol{s}_{n}(\widehat{\boldsymbol{\theta}_{n}})=\bm0$ with a probability tending to one as $n\rightarrow\infty$.
\end{theorem}

\begin{proof}
The proof follows from an appeal to \citet[Corollary~2.6]{Sorensen1999solutionExistence}. By condition ($\mathcal{C}.2$), $\boldsymbol{s}_{n}(\boldsymbol{\theta})$ is continuously differentiable with respect to $\boldsymbol\theta$ for all $\boldsymbol\theta \in \boldsymbol\Theta$. Thus, we only need to verify \citet[Condition~2.5]{Sorensen1999solutionExistence}. For elegance of writing, we omit the proof here but provide a proof in more general cases in Lemma~\ref{lemma:4_6}. 
\end{proof}

\section{Generalized infill asymptotic results}
\label{sec:extensions}

In this section, we discuss the extensions of our infill asymptotic results to general point processes and to other unbiased estimating equations that are based on the Campbell--Mecke theorem.

\subsection{Extensions to general point processes}
\label{subsec:generalpointprocess}

By releasing the Poisson assumption of $X_{n}$ and modifying corresponding conditions, the infill asymptotic results for logistic regression estimators obtained in Section \ref{sec:asymptotics} can be extended to other point processes, although a log-linear form (\ref{e:intensity}) for the intensity function may sometimes not be that natural. 

In the remainder of this section, we first demonstrate the extended infill asymptotics for general spatio-temporal point processes. Afterwards, we discuss the applicability for some specific families of point processes.

To extend the asymptotic results to general point processes, we shall need the following modified conditions:

\vspace{-\topsep}
\begin{description}[itemsep=0.1pt]
\item[$(\mathcal{C}.9)$] $Y_i$ is a point 
process with intensity function $\lambda(u; \boldsymbol\theta)$ given by (\ref{e:intensity}), where $b>0$ is an absolutely integrable function, $\bm{z}$ is a measurable vector of
covariates and the parameter vector $\boldsymbol \theta$ lies in an open set $\boldsymbol\Theta \subset \mathbb{R}^m$; $E_i$ has absolutely integrable intensity function $\rho(u) > 0$. 

\item[$(\mathcal{C}.10)$] Both $E_{i}$ and $Y_{i}$ have bounded pair correlation functions, $g$ and $h$,
that is, $\sup_{(u,v)\in (W\times T)^{2}}$ $g(u,v)$ $<\infty$ and $\sup_{(u,v)\in (W\times T)^{2}}h(u,v)<\infty$ .  
\end{description}
\vspace{-\topsep}

The theorems concerned with strong consistency, asymptotic normality and existence of the logistic regression estimator are then as follows.

\begin{theorem}
\label{theorem:consistency_general}
Assume that the conditions
($\mathcal{C}.1$),
($\mathcal{C}.4$)--($\mathcal{C}.7$) and ($\mathcal{C}.9$) hold. Define 
$l_n(\boldsymbol\theta) =
 l(X_n, D_n; \boldsymbol \theta)$ by
(\ref{e:LR_likelihood}) with 
$\boldsymbol\theta\in \boldsymbol\Theta$ and 
set 
$\widehat{\boldsymbol{\theta}_{n}}
= \argmax_{\boldsymbol\theta \in \boldsymbol{\Theta}}
l_{n}(\boldsymbol\theta)$. 
If $\widehat{\boldsymbol{\theta}_{n}}$ is attained,
then as $n\rightarrow \infty$,
$\widehat{\boldsymbol{\theta}_{n}}$ converges
$P_{\boldsymbol{\theta_0}}$-almost surely to
$\boldsymbol{\theta_0}$.
\end{theorem}
\begin{proof}
The same proofs as for Lemma \ref{lemma:4_1} and Theorem \ref{theorem:consistency} can be applied straightforwardly.
\end{proof}

\begin{theorem}
\label{theorem:asymptotics_general}
Assume that the conditions ($\mathcal{C}.1$), 
($\mathcal{C}.4$)--($\mathcal{C}.10$)  hold. 
Define $\boldsymbol{s}_{n}(\boldsymbol\theta)$ 
$= \boldsymbol{s}(X_n, D_n;$ $\boldsymbol\theta)$ by (\ref{e:LR_score_n}) with
$\boldsymbol\theta \in \boldsymbol\Theta$ and let $\widehat{\boldsymbol{\theta}_{n}}$ be the logistic regression estimator for which $\boldsymbol{s}_{n}(\widehat{\boldsymbol{\theta}_{n}})=\bm0$. If $\widehat{\boldsymbol{\theta}_{n}}$ is attained, then as $n\rightarrow\infty$, $n^{1/2}
(\widehat{\boldsymbol{\theta}_{n}}-
\boldsymbol{\theta_0})$ converges under 
$P_{\boldsymbol{\theta_0}}$ in
distribution to an
$m$-dimensional normally distributed random
vector with mean zero and 
covariance matrix $\bm{U}^{-1}\bm{V}(\bm{U}^{-1})^{\top}$, 
where $\bm{U}$ is as defined in 
Lemma~\ref{lemma:4_3} and $\bm{V}$ is now given by
\begin{equation*}
\begin{split}
    \bm{V} &= \left[
      \int_{W\times T}
          \frac{\lambda(u;\boldsymbol{\theta_0}) \rho(u)}
               {\lambda(u;\boldsymbol{\theta_0}) + \rho(u)}
      z^{(k)}(u) z^{(l)}(u) du
\right]_{k,l=1}^{m}\\
&+
\left[
     \int_{W\times T}\int_{W\times T}
     \frac{
       \lambda(u;\boldsymbol{\theta_0})
       \lambda(v;\boldsymbol{\theta_0})
       \rho(u)\rho(v) z^{(k)}(u) z^{(l)}(v) }
     {(\lambda(u;\boldsymbol{\theta_0}) + \rho(u))
     (\lambda(v;\boldsymbol{\theta_0}) + \rho(v))} \left(g(u,v)+h(u,v)-2\right)
     du dv \right]_{k,l=1}^{m}.
\end{split}
\end{equation*}
\end{theorem}

\begin{proof}
Similar proofs as for Lemma~\ref{lemma:4_3}, Lemma~\ref{lemma:4_4} and Theorem~\ref{lemma:4_5} can be used. Note that, by condition ($\mathcal{C}.10$), the pair correlation
function of $Y_i$, denoted by $h(u,v)$, is bounded, which
ensures the entries of $\bm{V}$ to be 
finite. 
\end{proof}

\begin{remark}
The covariance matrix now depends not only on the unknown intensity 
function $\lambda(u; \boldsymbol{\theta_0})$ but also on the pair correlation function $h(u,v)$. Similar to Remark~\ref{remark:practical_calculation}, in practice, a plug-in estimator of $\lambda(u; \boldsymbol{\theta_0})$ would be
used and the point processes $X_{n}$ and $D_{n}$ can be employed to
approximate the integrals.
Compared to Remark~\ref{remark:practical_calculation}, the approximation of $n \bm V$ needs the two extra terms below
\begin{equation*}
\left[
\sJM_{x,y\in X_{n}}
 \frac{\rho(x) z^{(k)}(x)}
 {\lambda(x;\widehat{\boldsymbol{\theta}_{n}})
   +\rho(x)}
 \frac{ \rho(y) z^{(l)}(y)}
 {\lambda(y;\widehat{\boldsymbol{\theta}_{n}})
 +\rho(y)} 
\right]_{k,l=1}^{m}
-
\left[
\sum_{x\in D_{n}}
\frac{ \lambda(x;\widehat{\boldsymbol{\theta}_{n}}) z^{(k)}(x)}
{\lambda(x;\widehat{\boldsymbol{\theta}_{n}})
+\rho(x)}
\sum_{y\in D_{n}}
\frac{ \lambda(y;\widehat{\boldsymbol{\theta}_{n}}) z^{(l)}(y)}
{\lambda(y;\widehat{\boldsymbol{\theta}_{n}})
+\rho(y)}
\right]_{k,l=1}^{m},
\end{equation*}
while the approximation of $n \bm U$ stays the same.
The construction of approximate confidence
intervals for the components of $\bm{\theta_0}$
is then straightforward again.
\end{remark}

\begin{theorem}
\label{theorem:asymptotic_existence_general}
Assume that the conditions ($\mathcal{C}.1$), 
($\mathcal{C}.4$)--($\mathcal{C}.10$)  hold. Define  $\boldsymbol{s}_{n}(\boldsymbol\theta)
= \boldsymbol{s}(X_n, D_n; \boldsymbol\theta)$ by (\ref{e:LR_score_n}) with
$\boldsymbol\theta \in \boldsymbol\Theta$.
Then, for every $n$, an estimator $\widehat{\boldsymbol{\theta}_{n}}$ exists that solves $\boldsymbol{s}_{n}(\widehat{\boldsymbol{\theta}_{n}})=\bm0$ with a probability tending to one as $n\rightarrow\infty$.
\end{theorem}

\begin{proof}
A similar proof as for Theorem \ref{theorem:asymptotics} can be applied as well.
\end{proof}

In application, for general point processes, the strength of interactions can sometimes enter into the intensity functions. 

For log-Gaussian Cox processes (\citealp{ColesJones1991,MollerSyversveenWaagepetersen}),
their intensity functions are the exponentials of Gaussian random variables driven by associated 
Gaussian random fields with mean function
$\left( \boldsymbol\theta^{\top} \bm{z}\left(u\right) \right) \log b\left(u\right)
$
and covariance function $\sigma^2 c(u,v)$. Here, $c(u,v)$ is
some correlation function 
and $\sigma^2 > 0$. Then, the intensity
functions read
$\lambda(u; \boldsymbol{\theta})\exp{(\sigma^2/2)} $,
which conforms with the log-linear form (\ref{e:intensity}) upon
adding a component with the entry one to $\boldsymbol z$. Should $c(u,v)$ depend on further
parameters, e.g. the decay rate of 
interactions, 
additional estimating equations are required. 

For Markov point processes \citep{Lieshout2000},
in general, their intensity functions are not
known explicitly. However, in the vector function (\ref{e:LR_base}),
one can replace the intensity function $\lambda(u;\boldsymbol\theta)$ by the 
Papangelou conditional intensity function $\lambda(u;\boldsymbol\theta | \bm{x})$
\citep{Baddeley2014LRPP}. The latter characterizes the probability of a new point occurring in an infinitesimal ball centred at a location and time combination $u \in W\times T$ given the current point realisations $\bm{x}$ of $X$. 
Similar infill asymptotics in such cases may be obtained upon adding appropriate conditions on the Papangelou conditional intensity functions.

\subsection{Extensions to other unbiased estimating equations}

The Campbell--Mecke theorem provides the theoretical foundation for a wide range of unbiased estimating equations. Although strong consistency of these estimators may not hold due to the loss of the likelihood interpretation, asymptotic normality and existence hold under appropriate conditions. 

In the remainder of this section, we extend the central limit theorem (cf., Theorems \ref{theorem:asymptotics_general} and \ref{theorem:asymptotic_existence_general}) to more general unbiased estimating equations by
introducing the necessary background and providing
full proofs. 

Recalling (\ref{e:campbell-mecke}), a more general unbiased estimating
function based on the Campbell--Mecke 
theorem can take the form
\begin{equation}
    \boldsymbol{s}(\boldsymbol\theta) = 
    \sum_{x\in X} \boldsymbol{f}(x;\boldsymbol\theta)
    -
    \sum_{x\in D} \boldsymbol{f}(x;\boldsymbol\theta)
    \frac{\lambda(x;\boldsymbol\theta)}{\rho(x)}
    =\boldsymbol{0},
    \label{e:general_score}
\end{equation}
for some test vector function
$\boldsymbol{f}: W \times T \to \mathbb{R}^m$ such that every component of $\boldsymbol{f}$, denoted by $f^{(i)}$ with $1 \leq i \leq m$, has the property that
$f^{(i)} \lambda$ is absolutely integrable. 
Note that, the test function $\boldsymbol{f}$ can contain some non-negative quadrature weight functions (e.g.,\ \citealp{Waagepetersen2008estimatingFunction}; \citealp{ Guan2010weightedEstimatingFunctions}) depending on the first- and second-order characteristics of $X$ to reduce the bias caused by deterministic numerical approximation.
 
To extend the central limit theorem to such unbiased estimating equations, we shall need the following additional conditions:

\vspace{-\topsep}
\begin{description}[itemsep=0.1pt]
\item[$(\mathcal{C}.11)$] The component functions $f^{(i)}(u,\boldsymbol\theta)$ are twice continuously
differentiable with respect to
$\boldsymbol\theta$. Its components and the first- and second-order
partial derivatives, denoted by $\frac{\partial}{\partial\theta^{(l)}}f^{(i)}(u; \boldsymbol{\theta})$ and $\frac{\partial^2}{\partial \theta^{(k)}
\partial \theta^{(l)}}f^{(i)}(u; \boldsymbol{\theta})$, in a neighbourhood of $\boldsymbol{\theta_0}$ are bounded in absolute value by some functions $d^{(i)}(u), d^{(i)}_l(u)$ and $d^{(i)}_{k,l}(u)$, respectively, that are absolutely integrable with respect to
$\lambda(u; \boldsymbol{\theta_0})$.

\item[$(\mathcal{C}.12)$]
The second-order partial derivatives satisfy an adapted H\H{o}lder condition in a neighbourhood of $\boldsymbol{\theta_0}$: there exist some $\alpha \in (0,1]$ and some functions $e^{(i)}_{k,l}(u)$ which are absolutely integrable with respect to
$\lambda(u; \boldsymbol{\theta_0})$ such that 
\begin{equation*}
    \left|
    \frac{\partial^2}{\partial \theta^{(k)} 
    \partial \theta^{(l)}}f^{(i)}(u; \boldsymbol{\theta})-
    \frac{\partial^2}{\partial \theta^{(k)} 
    \partial \theta^{(l)}}f^{(i)}(u; \boldsymbol{\theta_{0}})
    \right|\leq e^{(i)}_{k,l}(u)\|\boldsymbol\theta-\boldsymbol{\theta_{0}}\|^{\alpha}.
\end{equation*}

\item[$(\mathcal{C}.13)$] The matrix $\widetilde{\bm{U}}$, whose ($k,l$)-th entry reads
$\int_{W\times T}
f^{(k)}(u;\boldsymbol{\theta_0})
\lambda(u;\boldsymbol{\theta_0})
z^{(l)}(u)du$ with $\boldsymbol{\theta_{0}} \in \boldsymbol\Theta$ and
$1\leq k,l \leq m$,
is invertible.

\item[$(\mathcal{C}.14)$] The product of any two component functions, denoted by $f^{(k)}(u;\boldsymbol{\theta_0})f^{(l)}(u;\boldsymbol{\theta_0})$ with $1\leq k,l \leq m$, is absolutely integrable with respect to $\lambda(u; \boldsymbol{\theta_0})$.
\end{description}
\vspace{-\topsep}

A few remarks on the conditions above are appropriate. 
Condition ($\mathcal{C}.11$) and ($\mathcal{C}.12$) restrict the smoothness of $\bm{f}$ and its first-order and second-order derivatives. Condition ($\mathcal{C}.13$) provides a statement corresponding to condition ($\mathcal{C}.8$) but in the case of general unbiased estimating equations. Condition ($\mathcal{C}.14$) 
complements necessary technical constraints to ensure that the entries in the variances of these estimators are finite.

To establish the central limit theorem, we start our investigations with the following lemma. 

\begin{lemma}
\label{lemma:4_6}
Assume that the conditions ($\mathcal{C}.1$), ($\mathcal{C}.5$), ($\mathcal{C}.9$) and ($\mathcal{C}.11$) hold.
Define $\boldsymbol{s}_{n}(\boldsymbol\theta)
= \boldsymbol{s}(X_n, D_n; \boldsymbol\theta)$ by (\ref{e:general_score}) with
$\boldsymbol\theta \in \boldsymbol\Theta$.
Then, as $n\to\infty$, for all $\beta > 0$,
\begin{equation*}
    \sup_{\boldsymbol{\theta^{\prime,k}}\in C_{\beta}^{(n)}}\left\|-\frac{1}{n}\nabla \boldsymbol{s}_{n}(\boldsymbol{\theta^{\prime,1}},\dots,\boldsymbol{\theta^{\prime,m}})-\widetilde{\bm{U}}\right\|
\end{equation*}
converges under
$P_{\boldsymbol{\theta_0}}$ in probability to zero,
where 
$
    C_{\beta}^{(n)}=\{ \boldsymbol{\theta}\in \boldsymbol{\Theta} :  \|\boldsymbol{\theta}-\boldsymbol{\theta_{0}}\|\leq \beta/n^{1/2}\}
$
and
\begin{equation*}
        \widetilde{\bm{U}}=\left[\int_{W\times T}
        f^{(k)}(u;\boldsymbol{\theta_0})
        \lambda(u;\boldsymbol{\theta_0})
        z^{(l)}(u)  du
        \right]_{k,l=1}^{m}.
\end{equation*}
\end{lemma}

\begin{proof}
Under conditions ($\mathcal{C}.1$), ($\mathcal{C}.9$) and ($\mathcal{C}.11$) and recalling the unbiased estimating equation (\ref{e:general_score}),
$ - \nabla \boldsymbol{s}_{n}
 (\boldsymbol{\theta^{\prime,1}},
\dots, \boldsymbol{\theta^{\prime,m}})/n$ 
reads 
\begin{equation*}
\begin{split}
    &\frac{1}{n}
    \left[
    \sum_{x\in X_{n}} - \frac{\partial}{
    \partial \theta^{(l)}}  f^{(k)}(x;\boldsymbol{\theta^{\prime,k}})
    +\sum_{x\in D_{n}}
    \frac{\lambda(x; \boldsymbol{\theta^{\prime,k}})}
         {\rho(x)} 
    \left(
      \frac{\partial}{\partial \theta^{(l)}} f^{(k)} (x;\boldsymbol{\theta^{\prime,k}})
    + f^{(k)}(x;\boldsymbol{\theta^{\prime,k}})  z^{(l)}(x)
    \right)
 \right]_{k,l=1}^{m}.
\end{split}
\end{equation*}
Fix $\beta > 0$. We then prove the convergence of the supremum 
component-wisely. 

Firstly, by condition ($\mathcal{C}.11$), one can obtain that, for large enough $n$ and every $1 \leq i,k,l \leq m$,
\begin{equation}
    \left|
    f^{(i)}(u; \boldsymbol{\theta})
    \right| 
    \leq d^{(i)}(u),\quad
    \left|\frac{\partial}{\partial \theta^{(l)}} f^{(i)}(u;\boldsymbol{\theta}) \right|\leq
    d_l^{(i)}(u), \quad
    \left|\frac{\partial^2}{\partial \theta^{(k)} \partial \theta^{(l)}} f^{(i)}(u; \boldsymbol\theta) \right|\leq
    d^{(i)}_{k,l}(u),
    \label{e:bounds}
\end{equation}
for all $u\in W\times T$ and $\boldsymbol\theta \in C_{\beta}^{(n)}$. Note that, taking $n$ large enough, $C_\beta^{(n)}$ lies entirely within $\boldsymbol{\Theta}$, as the parameter space $\boldsymbol\Theta$ is open by condition ($\mathcal{C}.9$). 

Furthermore, under condition ($\mathcal{C}.9$), the first-order Taylor expansion of $\lambda(u; \boldsymbol{\theta})$ with respect to $\boldsymbol\theta$ at $\boldsymbol{\theta_{0}}$ reads
\begin{equation*}
    \lambda(u; \boldsymbol{\theta})
    -\lambda(u; \boldsymbol{\theta_{0}})=
     b\left(u\right) \exp\left( \boldsymbol{\underline{\theta}}^{\top}\bm{z}
     \right)
    \sum_{i=1}^{m}(\boldsymbol{\theta}-\boldsymbol{\theta_0})^{(i)}z^{(i)}(u),
\end{equation*}
where $\boldsymbol{\underline{\theta}}$ is a convex combination of $\boldsymbol{\theta}$ and $\boldsymbol{\theta_0}$. By condition ($\mathcal{C}.5$), the covariate terms in the inner product $\boldsymbol{\underline{\theta}}^{\top}\bm{z}$ and the sum above are all bounded. Thus, recalling the definition of $C_{\beta}^{(n)}$, when $n \to \infty$, one obtains that $\lambda(u; \boldsymbol{\theta})=\lambda(u; \boldsymbol{\theta_{0}})+b(u)o(1)$ for all $u\in W\times T$ and $\boldsymbol\theta \in C_{\beta}^{(n)}$. 

Now, consider $ - \nabla \boldsymbol{s}_{n}(\boldsymbol{\theta^{\prime,1}},\dots,\boldsymbol{\theta^{\prime,m}})/n$. Denote its $(k,l)$-th entry by $[-\nabla \boldsymbol{s}_{n}^{(k)}(\boldsymbol{\theta^{\prime,k}})/n]_{l}$ in the remainder of the proof. The difference between $[-\nabla \boldsymbol{s}_{n}^{(k)}(\boldsymbol{\theta^{\prime,k}})/n]_{l}$ and $[-\nabla \boldsymbol{s}_{n}(\boldsymbol{\theta_{0}})/n]_{k,l}$ reads
\begin{equation}
\begin{split}
    &- \frac{1}{n}
    \sum_{x\in X_{n}} \frac{\partial}{
    \partial \theta^{(l)}}  
    f^{(k)}(x;\boldsymbol{\theta^{\prime,k}})
    +\frac{1}{n}
    \sum_{x\in X_{n}} \frac{\partial}{
    \partial \theta^{(l)}}  
    f^{(k)}(x;\boldsymbol{\theta_{0}})
     \\
    &+
    \frac{1}{n}
    \sum_{x\in D_{n}}\frac{\lambda(x; \boldsymbol{\theta^{\prime,k}})}
         {\rho(x)} 
      \frac{\partial}{\partial \theta^{(l)}} 
      f^{(k)} (x;\boldsymbol{\theta^{\prime,k}})
      -\frac{1}{n}
    \sum_{x\in D_{n}} \frac{\lambda(x; \boldsymbol{\theta_{0}})}
         {\rho(x)} 
      \frac{\partial}{\partial \theta^{(l)}} 
      f^{(k)} (x;\boldsymbol{\theta_{0}})\\
    &+\frac{1}{n}
    \sum_{x\in D_{n}} 
    \frac{\lambda(x; \boldsymbol{\theta^{\prime,k}})}
         {\rho(x)} 
    f^{(k)}(x;\boldsymbol{\theta^{\prime,k}})z^{(l)}(x) - 
    \frac{1}{n}
    \sum_{x\in D_{n}} \frac{\lambda(x; \boldsymbol{\theta_{0}})}
         {\rho(x)} 
    f^{(k)}(x;\boldsymbol{\theta_{0}})z^{(l)}(x).
\end{split}
\label{e:differences}
\end{equation}
Below, we specifically analyze the asymptotic behavior of the second line in (\ref{e:differences}). The analysis of the other two lines proceeds along similar lines. 

Under condition ($\mathcal{C}.11$), the first-order Taylor expansion of $ \frac{\partial}{\partial \theta^{(l)}} 
f^{(k)} (x;\boldsymbol{\theta^{\prime,k}})$ with respect to $\boldsymbol\theta$ at $\boldsymbol{\theta_{0}}$ reads
\begin{equation*}
    \frac{\partial}{\partial \theta^{(l)}} 
    f^{(k)} (x;\boldsymbol{\theta^{\prime,k}})
    -\frac{\partial}{\partial \theta^{(l)}} 
    f^{(k)} (x;\boldsymbol{\theta_{0}})=
    \sum_{i=1}^{m}(\boldsymbol{\theta^{\prime,k}}-\boldsymbol{\theta_0})^{(i)}\frac{\partial^{2}}{
    \partial \theta^{(i)} \partial \theta^{(l)}} f^{(k)}(x;\boldsymbol{\theta^{\prime\prime,k}}),
\end{equation*}
where $\boldsymbol{\theta^{\prime\prime,k}}$ is a convex combination of $\boldsymbol{\theta^{\prime,k}}$ and $\boldsymbol{\theta_0}$, and thus $ \boldsymbol{\theta^{\prime\prime,k}}\in C_{\beta}^{(n)}$. Then, recalling that $\lambda(x; \boldsymbol{\theta^{\prime,k}})=\lambda(x; \boldsymbol{\theta_{0}})+b(x)o(1)$, the second line in (\ref{e:differences}) is bounded in absolute value by
\begin{equation*}
    \frac{1}{n}
    \sum_{x\in D_{n}} 
    \left\{
    \sum_{i=1}^{m}
    \frac{\lambda(x;\boldsymbol{\theta_{0}})}{\rho(x)}
    \left|(\boldsymbol{\theta^{\prime,k}}-\boldsymbol{\theta_0})^{(i)}\right|
    \left|\frac{\partial^{2}}{
    \partial \theta^{(i)} \partial \theta^{(l)}} f^{(k)}(x;\boldsymbol{\theta^{\prime\prime,k}})
    \right|+
    \left|\frac{b(x)o(1)}{\rho(x)}\right|
    \left|\frac{\partial}{
    \partial \theta^{(l)}} f^{(k)}(x;\boldsymbol{\theta^{\prime,k}})\right|
    \right\},
\end{equation*}
which, recalling the definition of $C_{\beta}^{(n)}$ and the bound functions in (\ref{e:bounds}), is further bounded by
\begin{equation}
    \sum_{i=1}^{m}
    \sum_{x\in D_{n}} 
    \frac{\beta}{n^{3/2}}
    \frac{\lambda(x;\boldsymbol{\theta_{0}})}{\rho(x)}
    d_{i,l}^{(k)}(x)
    +
    \sum_{x\in D_{n}} 
    \frac{1}{n}
    \frac{b(x)}{\rho(x)}|o(1)|
    d_{l}^{(k)}(x)
    \label{e:bound}
\end{equation}
for large enough $n$. Note that this bound depends only on $n$ and not on $\boldsymbol{\theta^{\prime,k}}$. Furthermore, by condition ($\mathcal{C}.11$), one can apply the Campbell-Mecke theorem and Markov's inequality as in the proof of Theorem \ref{lemma:4_5} to obtain that the inner sum in the first term of (\ref{e:bound}) converges under $P_{\boldsymbol{\theta_{0}}}$ in probability to zero. Since this result applies to every component of the inner sum with $1 \leq i \leq m$ and to the second term, (\ref{e:bound}) converges under $P_{\boldsymbol{\theta_{0}}}$ in probability to zero. Hence,
\begin{equation*}
    \sup_{\boldsymbol{\theta^{\prime,k}}\in C_{\beta}^{(n)}}
    \left|\frac{1}{n}
    \sum_{x\in D_{n}}
    \frac{\lambda(x; \boldsymbol{\theta^{\prime,k}})}
         {\rho(x)} 
      \frac{\partial}{\partial \theta^{(l)}} 
      f^{(k)} (x;\boldsymbol{\theta^{\prime,k}})
      - \frac{1}{n}
    \sum_{x\in D_{n}}\frac{\lambda(x; \boldsymbol{\theta_{0}})}
         {\rho(x)} 
      \frac{\partial}{\partial \theta^{(l)}} 
      f^{(k)} (x;\boldsymbol{\theta_{0}})
     \right|
    =o_{P}(1).
\end{equation*}

Applying similar proofs to the other two lines in (\ref{e:differences}) and combining the obtained results as above, by the triangle inequality, one obtains that
\begin{equation}
    \sup_{\boldsymbol{\theta^{\prime,k}}\in C_{\beta}^{(n)}}
    \left|\left[-\frac{1}{n}\nabla \boldsymbol{s}_{n}^{(k)}(\boldsymbol{\theta^{\prime,k}})\right]_{l}-\left[-\frac{1}{n}\nabla \boldsymbol{s}_{n}(\boldsymbol{\theta_{0}})\right]_{k,l}\right|
    =o_{P}(1).
    \label{e:lemma_sup_1}
\end{equation}

Finally, consider the analogue of Lemma
\ref{lemma:4_3} for the unbiased 
estimating equation (\ref{e:general_score}). Under conditions ($\mathcal{C}.1$), ($\mathcal{C}.9$) and ($\mathcal{C}.11$),
$-\nabla \boldsymbol{s}_{n}(\boldsymbol{\theta_0})/n$ now
reads 
\begin{equation*}
\begin{split}
    &\frac{1}{n}\sum_{i=1}^{n}
    \left[
    \sum_{x\in Y_{i}} - \frac{\partial}{
    \partial \theta^{(l)}}  f^{(k)}(x;\boldsymbol{\theta_0})
    +\sum_{x\in E_{i}}
    \frac{\lambda(x; \boldsymbol{\theta_0})}
         {\rho(x)} 
    \left(
      \frac{\partial}{\partial \theta^{(l)}} f^{(k)} (x;\boldsymbol{\theta_0})
    + f^{(k)}(x;\boldsymbol{\theta_0})  z^{(l)}(x)
    \right)
 \right]_{k,l=1}^{m}.
\end{split}
\end{equation*}
Again, by conditions ($\mathcal{C}.5$)
and ($\mathcal{C}.11$),
one can apply the Campbell--Mecke theorem and
Kolmogorov’s strong law of large numbers 
to obtain that 
$-\nabla \boldsymbol{s}_{n}(\boldsymbol{\theta_0})/n$ converges almost
surely, and
thus in probability, under $P_{\boldsymbol{\theta_0}}$ to $\widetilde{\bm{U}}$. Accordingly,
\begin{equation}
    \left|\left[-\frac{1}{n}\nabla \boldsymbol{s}_{n}(\boldsymbol{\theta_{0}})\right]_{k,l}-\widetilde{\bm{U}}_{k,l}\right|=o_{P}(1).
    \label{e:lemma_2}
\end{equation}
The proof is completed by applying the triangle inequality to combine (\ref{e:lemma_sup_1}) and (\ref{e:lemma_2}).
\end{proof}

The next theorem is then concerned with the central limit theorem.

\begin{theorem}
Assume that the conditions ($\mathcal{C}.1$), ($\mathcal{C}.4$)--($\mathcal{C}.5$) and ($\mathcal{C}.9$)--($\mathcal{C}.14$) hold. 
Define $\boldsymbol{s}_{n}(\boldsymbol\theta)$ by
(\ref{e:general_score}) with $\boldsymbol\theta\in \boldsymbol\Theta$.
Then, as $n\rightarrow\infty$,
for every $n$, an estimator $\widehat{\boldsymbol{\theta}_{n}}$ exists that solves
$\boldsymbol{s}_{n}(\boldsymbol\theta)=0$ with a probability tending to one. Moreover,
$\widehat{\boldsymbol{\theta}_n}$ converges 
under $P_{\boldsymbol{\theta_0}}$ in probability
to $\boldsymbol{\theta_0}$ and
$n^{1/2}(\widehat{\boldsymbol{\theta}_n}
-\boldsymbol{\theta_0})$ converges under 
$P_{\boldsymbol{\theta_0}}$ in
distribution to an $m$-dimensional normally
distributed random vector with mean 
$\boldsymbol 0$ and covariance matrix
$\widetilde{\bm{U}}^{-1}\widetilde{\bm{V}}(\widetilde{\bm{U}}^{-1})^{\top}$, where
$\widetilde{\bm{U}}$ is as defined in Lemma \ref{lemma:4_6} and $\widetilde{\bm{V}}$ is given by
\begin{equation*}
    \begin{split}
        \widetilde{\bm{V}}
        &=\left[\int_{W\times T} f^{(k)}(u;\boldsymbol{\theta_0})
        f^{(l)}(u; \boldsymbol{\theta_0})
        \lambda(u;\boldsymbol{\theta_0})
        \left( 1 + 
        \frac{\lambda(u;\boldsymbol{\theta_0})}
        {\rho(u)} \right) du
        \right]_{k,l=1}^{m}\\ 
        &+\left[\int_{W\times T}\int_{W\times T} f^{(k)}(u;\boldsymbol{\theta_0})
        \lambda(u;\boldsymbol{\theta_0})
        f^{(l)}(v;\boldsymbol{\theta_0})
        \lambda(v;\boldsymbol{\theta_0}) 
        (g(u,v)+h(u,v)-2 )dudv\right]_{k,l=1}^{m}.
    \end{split}
    \end{equation*}
\end{theorem}

\begin{proof}
First of all, by condition ($\mathcal{C}.11$), $\boldsymbol{s}_{n}(\boldsymbol{\theta})$ is twice continuously differentiable with respect to $\boldsymbol\theta$. Then, we shall verify the analogues of Lemmas
\ref{lemma:4_3} and \ref{lemma:4_4} for the unbiased 
estimating equation (\ref{e:general_score}). 

The analogue of Lemma \ref{lemma:4_3} was established in Lemma~\ref{lemma:4_6}. For Lemma \ref{lemma:4_4},
under conditions ($\mathcal{C}.1$) and ($\mathcal{C}.9$),
$\boldsymbol{s}_{n}(\boldsymbol\theta)/n^{1/2}$ now reads
\begin{equation*}
    \frac{\boldsymbol{s}_{n}(\boldsymbol{\theta_0})}{n^{1/2}}
     =  n^{1/2}\left\{\frac{1}{n}\sum_{i=1}^{n}
    \sum_{x\in Y_{i}}
    \boldsymbol{f}(x;\boldsymbol{\theta_0})
    -\frac{1}{n}\sum_{i=1}^{n}\sum_{x\in E_{i}}
    \boldsymbol{f}(x;\boldsymbol{\theta_0})
     \frac{\lambda(x;\boldsymbol{\theta_0})}{\rho(x)}
    \right\}.
\end{equation*}
The two terms between the curly brackets
above are both averages of independent and identically
distributed real-valued random vectors. By conditions ($\mathcal{C}.9$)--($\mathcal{C}.11$), their means
are finite and identical. Additionally by conditions ($\mathcal{C}.4$) and ($\mathcal{C}.14$), their variances
are also finite and given by the entries of $\widetilde{\bm{V}}$.
Then, similar to the proof of Lemma \ref{lemma:4_4}, one can apply the Campbell--Mecke theorem and multi-variate Lindeberg–L\'evy’s central limit theorem to
obtain that $\boldsymbol{s}_{n}(\boldsymbol{\theta_{0}})/n^{1/2}$
converges under  $P_{\boldsymbol{\theta_0}}$
in distribution to an $m$-dimensional normally distributed random vector with mean zero and covariance matrix $\widetilde{\bm{V}}$.

From \citet[Condition~2.7, Corollary~2.8 and Theorem~2.9]{Sorensen1999solutionExistence} and by condition ($\mathcal{C}.13$), we thus only need to prove that 
\begin{equation*}
    \sup_{\boldsymbol{\theta}\in C_{\beta}^{(n)}}\left\|
    -\frac{1}{n}\boldsymbol{s}_{n}(\boldsymbol{\theta})\right\|, \quad 
    \sup_{\boldsymbol{\theta^{\prime,k}}\in C_{\beta}^{(n)}}\left\|-\frac{1}{n}\nabla \boldsymbol{s}_{n}(\boldsymbol{\theta^{\prime,1}},\dots,\boldsymbol{\theta^{\prime,m}})-\widetilde{\bm{U}}\right\|
\end{equation*}
and
\begin{equation*}
    \sup_{\boldsymbol{\theta^{\prime,k}}\in C_{\beta}^{(n)}}\left\|-\frac{1}{n}\nabla^{2} \boldsymbol{s}^{(i)}_{n}(\boldsymbol{\theta^{\prime,1}},\dots,\boldsymbol{\theta^{\prime,m}})-\widetilde{\bm{Q}}^{(i)}\right\|
\end{equation*}
all converge under $P_{\boldsymbol{\theta_0}}$ in probability to zero, where $\widetilde{\bm{Q}}^{(i)}$ is given by
\begin{equation*}
\begin{split}
    \widetilde{\bm{Q}}^{(i)}
    &=\left[\int_{W\times T}f^{(i)}(u,\boldsymbol{\theta_0})\lambda(u,\boldsymbol{\theta_0})z^{(k)}(u)z^{(l)}(u)du\right]_{k,l=1}^{m} \\
    & +\left[\int_{W\times T}\left(\frac{\partial}{\partial\theta^{(k)}}f^{(i)}(u,\boldsymbol{\theta_0})z^{(l)}(u)+\frac{\partial}{\partial\theta^{(l)}}f^{(i)}(u,\boldsymbol{\theta_0})z^{(k)}(u)\right)\lambda(u,\boldsymbol{\theta_0})du\right]_{k,l=1}^{m}.
\end{split}
\end{equation*}
In Lemma~\ref{lemma:4_6}, we already proved that
\begin{equation*}
    \sup_{\boldsymbol{\theta^{\prime,k}}\in C_{\beta}^{(n)}}\left\|-\frac{1}{n}\nabla \boldsymbol{s}_{n}(\boldsymbol{\theta^{\prime,1}},\dots,\boldsymbol{\theta^{\prime,m}})-\widetilde{\bm{U}}\right\| \stackrel{P_{\boldsymbol{\theta_0}}}{\to} 0.
\end{equation*}
The remainder of the proof mostly proceeds along similar lines for the other two cases. 

Note that condition ($\mathcal{C}.12$) is needed in proving the convergence of some terms in the $(k,l)$-th entry of $-\nabla^{2} \boldsymbol{s}^{(i)}_{n}(\boldsymbol{\theta^{\prime,1}},\dots,\boldsymbol{\theta^{\prime,m}})/n$. For instance, we need to prove that
\begin{equation*}
    \sup_{\boldsymbol{\theta^{\prime,k}}\in C_{\beta}^{(n)}}\left|\frac{1}{n}
    \sum_{x\in D_{n}}
    \frac{\lambda(x; \boldsymbol{\theta^{\prime,k}})}
         {\rho(x)} 
      \frac{\partial^{2}}{\partial \theta^{(j)}\theta^{(l)}} 
      f^{(k)} (x;\boldsymbol{\theta^{\prime,k}})
      - \frac{1}{n}
    \sum_{x\in D_{n}}\frac{\lambda(x; \boldsymbol{\theta_{0}})}
         {\rho(x)} 
      \frac{\partial^{2}}{\partial \theta^{(j)}\theta^{(l)}} 
      f^{(k)} (x;\boldsymbol{\theta_{0}})\right|=o_{P}(1).
\end{equation*}
Since $f^{(k)}(u,\boldsymbol\theta)$ is only twice continuously
differentiable with respect to
$\boldsymbol\theta$, the proof in Lemma \ref{lemma:4_6} does not apply here. However, recalling from the proof of Lemma \ref{lemma:4_6} that $\lambda(x; \boldsymbol{\theta^{\prime,k}})=\lambda(x; \boldsymbol{\theta_{0}})+b(x)o(1)$, the absolute value term above is bounded by
\begin{equation*}
    \frac{1}{n}
    \sum_{x\in D_{n}} 
    \left\{
    \frac{\lambda(x;\boldsymbol{\theta_{0}})}{\rho(x)}
    \left|\frac{\partial^{2} f^{(k)} (x;\boldsymbol{\theta^{\prime,k}})}{\partial \theta^{(j)}\theta^{(l)}} 
      -
      \frac{\partial^{2} f^{(k)} (x;\boldsymbol{\theta_{0}})}{\partial \theta^{(j)}\theta^{(l)}}
    \right|+
    \left|\frac{b(x)o(1)}{\rho(x)}\right|
    \left|\frac{\partial^{2} f^{(k)} (x;\boldsymbol{\theta^{\prime,k}})}{\partial \theta^{(j)}\theta^{(l)}} 
     \right|\right\},
\end{equation*}
which, under the adapted H\H{o}lder condition in condition ($\mathcal{C}.12$), is further bounded by 
\begin{equation*}
    \sum_{x\in D_{n}} 
    \left\{
    \frac{\beta}{n^{\alpha/2+1}}
    \frac{\lambda(x;\boldsymbol{\theta_{0}})}{\rho(x)}
    e_{i,l}^{(k)}(x)+
    \frac{1}{n}
    \frac{b(x)}{\rho(x)}\left|o(1)\right|
    d_{i,l}^{(k)}(x)
    \right\}
\end{equation*}
for large enough $n$. This bound depends only on $n$ and not on $\boldsymbol\theta^{\prime,k}$. By conditions ($\mathcal{C}.11$) and ($\mathcal{C}.12$), one can again apply the Campbell-Mecke theorem and Markov's inequality as in the proof of Theorem \ref{lemma:4_5} to obtain that the bound converges under $P_{\boldsymbol{\theta_{0}}}$ in probability to zero.
\end{proof}

\begin{remark}
Looking back on the logistic regression estimator given by (\ref{e:LR_base}), its component functions are smooth with respect to $\boldsymbol\theta$ and thus the second-order derivatives satisfy the adapted H\H{o}lder condition. Moreover, its component functions and the first- and second-order derivatives are bounded under condition ($\mathcal{C}.5$). 
\end{remark}

\section{Conclusion}
\label{sec:conclusion}
In this paper, we established strong consistency and asymptotic normality for logistic regression estimators for spatio-temporal point processes under the regime of infill asymptotics. We also extended the asymptotic results to other unbiased estimating functions that are based on the Campbell--Mecke theorem. We demonstrated how to construct approximate confidence intervals in practice. Compared to increasing-domain asymptotics, infill asymptotics provide the theoretical foundation on the asymptotic behaviour of parameter estimates where identical point patterns are only valid within a fixed region.

In addition, we assumed that spatio-temporal point processes of interest have log-linear parametric intensity functions, which simplifies the computation. However, we believe that the asymptotic results can be proved for a wide class of intensity functions. Moreover, motivated by \citet{Baddeley2014LRPP}, it would also be interesting to extend infill asymptotics to the Gibbs domain by studying associated background and conditions.


\section*{Acknowledgements}
This research was funded by the Dutch Research Council (NWO) for the project 
`Data driven risk management for fire services' (18004). We thank Maurits
de Graaf and Cl\'{e}ment Lezane for valuable discussions and comments. 

\bibliographystyle{plainnat}
\bibliography{bibliography.bib}

\end{document}